\documentclass{amsart}
\usepackage{amsmath}
\usepackage{amssymb}
\usepackage{amsthm}
\usepackage{lineno}
\usepackage{subfigure}
\usepackage{graphicx}
\usepackage{color}
\usepackage{xspace}
\usepackage{hyperref} 
\graphicspath{{../MAfigures/}}

\newcommand{\mao}{MA}
\newcommand{\delu}{v}
\newcommand{\trace}{\text{trace}}

\newcommand{\bq}{\begin{equation}}
\newcommand{\eq}{\end{equation}}
\newcommand{\R}{\mathbb{R}}
\newcommand{\dm}{d}
\newcommand{\Rd}{\R^\dm}
\newcommand{\Dt}{\mathcal{D}}
\newcommand{\abs}[1]{\left\vert#1\right\vert}
\newcommand{\norm}[1]{\left\vert#1\right\vert}
\newcommand{\DD}[2]{\frac{\partial^2 #1}{\partial {#2}^2}}
\newcommand{\MA}{{Monge-Amp\`ere}\xspace}

\newcommand{\grad}{\nabla}
\newcommand{\e}{\epsilon}
\newcommand{\ex}[1]{\times 10^{- #1}}
\newcommand{\G}{\mathcal{G}}
\newcommand{\bO}{\mathcal{O}}

\newcommand{\xv}{\mathbf{x}}
\newcommand{\xo}{\mathbf{x}_0}

\newtheorem{theorem}{Theorem}
\theoremstyle{definition}
\newtheorem{lemma}[theorem]{Lemma}
\newtheorem{definition}{Definition}
\theoremstyle{remark}
\newtheorem{example}{Example}
\newtheorem{remark}{Remark}

\begin{document}

\title[Finite difference solvers for the Monge-Amp\`ere equation]{Convergent finite difference solvers for viscosity solutions of the elliptic Monge-Amp\`ere equation in dimensions two and higher}

\author{Brittany D. Froese \and Adam M. Oberman}
\address{{Department of Mathematics, Simon Fraser University\\ Burnaby, British Columbia, Canada, V5A 1S6}}
\email{aoberman,bdf1@sfu.ca}
%%%%%%%%%%%%%%%%%%%%

\begin{abstract}
The elliptic Monge-Amp\`ere equation is a fully nonlinear Partial Differential Equation  that originated in geometric surface theory and has been applied in dynamic meteorology, elasticity, geometric optics, image processing and image registration.   Solutions can be singular, in which case standard numerical approaches fail.  Novel solution methods are required for stability and convergence to the weak (viscosity) solution.

In this article we build a wide stencil finite difference discretization for the \MA equation.  The scheme is monotone, so the Barles-Souganidis theory allows us to prove that the solution of the scheme converges to the unique viscosity solution of the equation.

Solutions of the scheme are found using a damped Newton's method.
We prove convergence of Newton's method and provide a systematic method to determine a starting point for the Newton iteration.

Computational results are presented in two and three dimensions, which demonstrates the speed and accuracy of the method on a number of exact solutions, which range in regularity from smooth to non-differentiable.
\end{abstract}

\date{\today}    
%\subjclass[1991]{}
\keywords{Fully Nonlinear Elliptic Partial Differential Equations, Monge Amp\`ere equations, Nonlinear Finite Difference Methods, Viscosity Solutions, Monotone Schemes, Convexity Constraints}
\maketitle

% !TEX root = NMATmain.tex

\section{Introduction}
In this article we introduce a monotone discretization of the \MA equation, which is valid in arbitrary dimensions.  A proof of convergence to the viscosity solution is provided, as well as a proof of convergence of Newton's method.   Numerical results are also presented.

\subsection{The setting for equation}

The \MA equation is a fully nonlinear Partial Differential Equation (PDE).  
\bq
\label{MA}\tag{MA}
\det(D^2u(x)) = f(x), \quad \text{for $x$ in }\Omega.
\eq
The \MA operator, $\det(D^2u)$,  is the determinant of the Hessian of the function $u$.   The equation is augmented by the convexity constraint
\bq\label{convex}\tag{C}
u \text{ is convex, }
\eq
which is necessary for the equation to be elliptic.  
The convexity constraint is made explicit for emphasis; it is necessary for uniqueness of solutions and it is essential for numerical stability.  

While other boundary conditions appear naturally in applications, we consider the simplest boundary conditions: the Dirichlet problem in a convex bounded subset~$ \Omega \subset \Rd$ 
with boundary~$\partial\Omega$,
\bq\tag{D}\label{Dirichlet}
u(x) = g(x), \quad \text{for $x$ on }\partial\Omega.
\eq
Under suitable assumptions on the domain and the functions $f(x), g(x)$, recalled in~\autoref{sec:regularity}, there exists a unique classical ($C^2$) solution to~\eqref{MA},~\eqref{convex},~\eqref{Dirichlet}.  
However, when these conditions fail, the solution can be singular.  
For singular solutions, the correct notion of weak solution must be used.
In this case, novel discretizations and solution methods must be used to approximate the solution.

\subsection{Applications}
The PDE~\eqref{MA} is a geometric equation, which goes back to Monge and Amp\`ere (see~\cite{EvansSurvey}).   
The equation naturally arises in geometric problems:
the Minkowski problem for prescribing the Gaussian curvature of a surface~\cite{Bakelman, PogBook, PrescribingCurvature} and the related isometric embedding problem for manifolds~\cite{IsometricEmbedding}.
 Early applications identified in~\cite{olikerprussner88} include dynamic meteorology, elasticity, and geometric optics.%~\cite{Haltiner, Kasahara, Stoker, Westcott}.  

The \MA equation arises as the optimality condition for the problem of optimal mass transport with quadratic cost~\cite{EvansSurvey, Ambrosio,Villani}.
This application of the \MA equation has been used in many areas:
image registration~\cite{Haker, HakerRegistration, RehmanRegistration},
 mesh generation~\cite{Delzanno, DelzannoGrid, Budd}, 
 reflector design~\cite{GlimmOlikerReflectorDesign}, 
 and astrophysics (estimating the shape of the early universe)~\cite{FrischUniv}.

\newcommand{\fv}{\mathbf{g}}

The optimal transportation problem seeks a map $\fv(x)$ that moves the measure $\mu_1(x)$ to $\mu_2(y)$ and minimizes the transportation cost functional
\[ 
\int_{\R^d} \norm{ x-\fv(x) }^2 \,d\mu_1. 
\]
The optimal mapping is given by $\fv = \grad u$,  where $u$ is convex and satisfies the \MA equation
\[ \det(D^2u(x)) = \mu_1(x)/\mu_2(\nabla u(x)). \]  
In this situation, the Dirichlet boundary condition~\eqref{Dirichlet} is replaced by the implicit boundary condition
\bq\label{BC2}
\fv(\cdot) : \Omega_1 \to \Omega_2,
\eq
where the sets $\Omega_1, \Omega_2$ are the support of the measures~$\mu_1, \mu_2$, respectively.
These boundary conditions are implicit and thus difficult to implement numerically.  
For many applications,  both domains are squares, and a simplifying assumption that edges are mapped to edges allows {Neumann} boundary conditions to be used.  In other applications, periodic boundary conditions are used.
Building on the work in this article, the first author has recently devised a method to implement the boundary condition~\eqref{BC2}~\cite{FroeseMATransport}.

In other problems, the \MA operator appears in an \emph{inequality constraint} in a variational problem for optimal mappings where the cost is no longer the transportation cost.  Here the operator has the effect of restricting the local area change on the set of admissible mappings; see~\cite{HaberConstrained} or~\cite{CohenOr}.  

\subsection{Related numerical works}
Despite the number of applications, few publications devoted to the numerical solution of the \MA equation have appeared until recently.  We make a distinction between numerical approaches with optimal transportation type boundary conditions~\eqref{BC2} and the standard Dirichlet boundary conditions~\eqref{Dirichlet}.  
In the latter case, a number of numerical methods have recently been proposed for the solution of the \MA equation.

The early work by Benamou and Brenier~\cite{BenBren} used a fluid mechanical approach to compute the solution to the optimal transportation problem.
The work by Oliker and Prussner~\cite{olikerprussner88} presents a discretization that converges to the Aleksandrov solution in two dimensions.  This is a very early method which solves a problem with only about a dozen grid points.  

For the Dirichlet problem treated herein, a series of papers have recently appeared by two groups of authors, Dean and Glowinski~\cite{DGnum2008, DGaug, GlowinksiICIAM} and Feng and Neilan~\cite{FengMA, FengFully}.  The methods introduced by these authors perform best in the regular case and can break down in the singular case.  See~\cite{BeFrObMA} for a more complete discussion.

We also mention the works~\cite{LoeperMA}, in the periodic case, and~\cite{Delzanno} for applications to mappings.  The method of~\cite{MAFrisch} treats the problem of periodic boundary conditions in an odd-dimensional space.

\subsection{Numerical challenges}

%\subsubsection*{Standard methods for regular solutions}
When the conditions for regularity are satisfied, classical solutions can be approximated successfully using a range of standard techniques (see, for example, works such as~\cite{DGaug, DGnum2008, GlowinksiICIAM} and~\cite{FengMA, FengFully}).  In an earlier work~\cite{BeFrObMA}, we studied a simple finite difference discretization, which was accurate and fast on smooth solutions.
However, for singular solutions, standard numerical methods break down: either by becoming unstable, poorly conditioned, or by selecting the wrong (non-convex) solution.

\subsubsection*{Weak solutions}
For singular solutions, the appropriate notion of weak (viscosity or Aleksandrov) solution must be used.  
Numerical methods have been developed which capture weak solutions:
Oliker and Prussner, in~\cite{olikerprussner88},  presented a method that converges to the Aleksandrov solution.  
The second author~\cite{ObermanEigenvalues} introduced a wide stencil finite difference method which converges to the viscosity solution. Both of these methods were restricted to two dimensions.    %However, methods which are provably convergent may have lower accuracy or slower solution methods than other methods, which are effective for regular solutions.

\subsubsection*{Convexity}
The convexity constraint is necessary for both uniqueness and stability.
In particular, the equation~\eqref{MA} fails to be elliptic if $u$ is non-convex (see~\autoref{sec:ellipt}), so instabilities can arise if the convexity condition~\eqref{convex} is violated.
Any approximation of~\eqref{MA} requires some selection principle to choose the convex solution.  
This selection principle can be built into the discretization, as in~\cite{ObermanEigenvalues}, or built into the solution method, as in~\cite{BeFrObMA}.

\subsubsection*{Accuracy}
The convergent monotone scheme of~\cite{ObermanEigenvalues} uses a wide stencil and the accuracy of the scheme depends on the \emph{directional resolution}, which in turn depends on the width of the stencil.  As we demonstrate below, for highly singular solutions such as~\eqref{eq:cone}, the directional resolution error can dominate.  However, it is unrealistic to expect high accuracy for singular solutions.   Experimental results on singular solutions using standard finite differences, which is formally $\bO(h^2)$, produce results which are only accurate to $\bO(\sqrt h)$; see~\cite{BeFrObMA}.

\subsubsection*{Fast solvers}
Previous work by the authors and a coauthor~\cite{BeFrObMA} investigated fast solvers for~\eqref{MA}.  An explicit method was presented which was moderately fast, independent of the solution time.  
For regular solutions, a faster (by an order of magnitude) semi-implicit solution method was introduced (see~\autoref{sec:SI}),  but this method was slower (by an order of magnitude) on singular solutions.

% !TEX root = NMATmain.tex

\section{Analysis and weak solutions}\label{sec:weak}
In this section we review regularity results and background analysis.

Example solutions include~\eqref{eq:c1} and~\eqref{eq:cone}, below.  The first is a viscosity solution, defined in \autoref{sec:viscosity}.  The second is an Aleksandrov solution, defined in \autoref{sec:alex}.   

The regularity conditions required for classical solutions are reviewed in the section below.  The weaker notion of solution, viscosity solutions, allows the continuous function $f$ to be zero.  The even weaker geometric notion of Aleksandrov solutions allows $f$ to be a nonnegative measure.

\subsection{Regularity}\label{sec:regularity}
Regularity results for the \MA equation have been established in~\cite{CafNirSpruck,UrbasDirichletMA,CaffEstimates}. We refer to the book~\cite{Gutierrez} for the following result.

Under the following conditions, the \MA equation is guaranteed to have a unique $C^{2,\alpha}$ solution.
\bq\label{conditions}
\begin{cases}
\text{The domain, $\Omega$, is strictly convex with boundary $\partial\Omega\in C^{2,\alpha}$.}\\
\text{The boundary values, $g \in C^{2,\alpha}(\partial\Omega)$.}\\
\text{The function, $f \in C^\alpha(\Omega)$ is strictly positive.}
\end{cases}
\eq

\begin{remark}  In the extreme case, with $f(x) = 0$ for all $x \in \Omega$, the equation~\eqref{MA},\eqref{convex} reduces to the computation of the convex envelope of the boundary conditions~\cite{ObermanCENumerics, ObSilvConvEnv}.  In this case, solutions may not even be continuous up to the boundary and can also 
be non-differentiable in the interior.
\end{remark}

\begin{remark}
While it is usual to perform numerical solutions on a rectangle, regularity can break down in particular convex polygons~\cite{Villani,Pogorelov}.  
\end{remark}

\subsection{Viscosity solutions}\label{sec:viscosity}
We recall the definition of a viscosity solution~\cite{CIL}, which is defined for the \MA equation in~\cite{Gutierrez}. 
\begin{definition}
Let $u \in C(\Omega)$ be convex and $f\geq0$ be continuous.  The function $u$ is a \emph{viscosity subsolution (supersolution)} of the \MA equation in $\Omega$ if whenever convex $\phi\in C^2(\Omega)$ and $x_0\in\Omega$ are such that $(u-\phi)(x)\leq(\geq)(u-\phi)(x_0)$ for all $x$ in a neighborhood of $x_0$, then we must have
\[ \det(D^2\phi(x_0)) \geq(\leq)f(x_0). \]
The function $u$ is a \emph{viscosity solution} if it is both a viscosity subsolution and supersolution.
\end{definition}

\begin{example}[Viscosity solution]
This example will be computed in sections~\ref{sec:2d}-\ref{sec:3d}.  
Consider~\eqref{MA} in two dimensions, with solution, $u$,  and data, $f$, given by
\[ 
u(\xv) = \frac{1}{2}  \left ((\abs{\xv}-1)^+ \right)^2, 
\qquad
f(\xv) = \left (1-  \frac 1 {\abs{\xv}} \right)^+. 
\]
Here  we use the notation $h^+ = \max(h,0)$.
(In three dimensions, the function $f$ is different; see \autoref{sec:3d}).
The function $u$ is a viscosity solution but not a classical $C^2$ solution of the \MA equation.

We verify the definition of a viscosity solution.  This only needs to be done at points where $\abs{\xo}=1$ (since $u$ is locally $C^2$ away from this circle).  We note that $f$ is equal to zero on this circle.

We begin by checking convex $C^2$ functions $\phi \leq u$ with $\phi(\xo) = u(\xo) = 0$ (that is, $u-\phi$ has a local minimum here).  
Since $\nabla u(\xo) = 0$, we require $\nabla \phi(\xo) = 0$ as well.  Since $u$ is constant in part of any neighborhood of $\xo$, any convex $\phi$ must also be constant in this part of the neighborhood in order to ensure that $u-\phi$ has a local minimum.  This means that $\phi$ has zero curvature in some directions so that $\det D^2 \phi(\xo) = 0$, as required by the definition of the viscosity solution.  We conclude that $u$ is a supersolution of the \MA equation.

We also need to check functions $\phi \geq u$ with $\phi(\xo) = u(\xo) = 0$ (so that $u-\phi$ has a local maximum).  Since $\phi$ is convex, it will automatically satisfy the condition $\det D^2\phi(\xo) \geq 0$ and we conclude that $u$ is a subsolution. 
\end{example}

\subsection{Aleksandrov solutions}\label{sec:alex}
Next we turn our attention to the Aleksandrov solution, which is a more general weak solution than the viscosity solution.
Here $f$ is generally a measure~\cite{Gutierrez}.  We begin by recalling the definition of the normal mapping or subdifferential of a function.

\begin{definition}
The \emph{normal mapping} (\emph{subdifferential}) of a function $u$ is the set-valued function $\partial u$ defined by
\[ 
\partial u(x_0) = \left \{p  \mid u(x) \geq u(x_0) + p\cdot (x-x_0), \text{ for all } x\in\Omega \right \}. 
\]
For a set $V \subset\Omega$, we define $\partial u(V) = \bigcup\limits_{x\in V}\partial u(x)$.
\end{definition}

Now we want to look at a measure generated by the \MA operator.  
%Of course, the most natural choice is simply
%\[ \mu(V) = \int_V \det (D^2u(x))\,dx. \]
%However, this is only defined if the function $u\in C^2(\Omega)$.  We generalize this to solutions with less regularity as follows:
\begin{definition}
Given a function $u\in C(\Omega)$, the \emph{\MA measure} associated with $u$ is defined as
\[ \mu(V) = \abs{\partial u(V)}\]
for any set $V \subset\Omega$, where $\abs{B}$ is the measure of the set $B$.
\end{definition}

This measure naturally leads to the notion of the generalized or Aleksandrov solution of the \MA equation.
\begin{definition}
Let $\mu$ be a Borel measure defined in a convex set $\Omega\in\R^d$.  The convex function $u$ is an \emph{Aleksandrov solution} of the \MA equation
\[ \det(D^2u) = \mu \]
if the \MA measure associated with $u$ is equal to the given measure $\mu$.
\end{definition}

\begin{example}[Aleksandrov solution]
As an example, we consider the cone and the the scaled Dirac measure
\[ u(\xv) = \norm{\xv}, 
\qquad
 \mu(V) = \pi \int_V \delta(\xv)\,d\xv. 
 \]
We verify from the definition that $u,\mu$ is an Aleksandrov solution of the \MA equation. (Since $\mu$ is a measure, we cannot interpret $u$ as a viscosity solution of \MA.)  The subdifferential $\partial u$ is given by
\[ 
\partial u(\xv) = 
\begin{cases} 
\xv/ \norm{\xv}, & \norm{\xv}>0,
\\
B_1, & \xv = \mathbf{0},
\end{cases} 
\]
where $B_1 = \left \{\xv \in \R^d \mid \norm{\xv} \leq 1\right\}$.
Then the associated \MA measure will be
\[ \abs{\partial u(V)} = 
\begin{cases} \pi & \mathbf{0}\in V\\ 0 & \mathbf{0} \notin V\end{cases}
= \pi \int_V \delta(x)\,dx = \mu(V).
 \]
\end{example}

\subsection{A PDE for convexity}
The convexity constraint~\eqref{convex} is necessary for uniqueness. 
For example, in two dimensions,  ignoring boundary conditions, in the absence of the convexity constraint, $-u$ will also be a solution of~\eqref{MA} whenever $u$ is a solution.

For a twice continuously differentiable function $u$, the convexity restriction~\eqref{convex} is equivalent to requiring that the Hessian, $D^2u$, be positive definite. 
Since we wish to work with less regular solutions,~\eqref{convex} can be enforced by the equation
\[
\lambda_1(D^2u) \ge 0,
\] 
understood in the viscosity sense~\cite{ObermanCENumerics, ObSilvConvEnv}, 
where $\lambda_1 (D^2u)$ is the smallest eigenvalue of the Hessian of $u$.

The convexity constraint can be absorbed into the operator by
 defining
\bq\label{detplus}
{\det}^+(M) = \prod\limits_{j=1}^d\lambda_j^+ 
\eq
where $M$ is a symmetric matrix, with eigenvalues, $\lambda_1 \le \dots, \le \lambda_n$ and
\[
x^+ = \max(x, 0).
\] 
Using this notation,~\eqref{MA},\eqref{convex} becomes
\bq
\label{MAplus}\tag*{$(MA)^+$}
{\det}^+(D^2u(x))= f(x), \quad \text{for $x$ in }\Omega.
\eq

\begin{remark}
Notice that there is a trade-off in defining~\eqref{detplus}: the constraint~\eqref{convex} is eliminated but the operator becomes non-differentiable near singular matrices.
\end{remark}

\subsection{Linearization and ellipticity}\label{sec:ellipt}
The linearization of the determinant is given by
\[
\grad \det(M) \cdot N = 
\trace 
\left(   
	M_{adj} N	
\right )
\]
where $M_{adj}$ is the adjugate~\cite{Strang}, which is the transpose of the cofactor matrix.  The adjugate matrix is positive definite if and only if $M$ is positive definite.
When the matrix $M$ is invertible, the adjugate $M_{adj}$ satisfies
\bq\label{adjinv}
M_{adj} = \det(M) M^{-1}.
\eq

We now apply these considerations to the linearization of the \MA operator.
When $u \in C^2$ we can linearize this operator as
\bq\label{eq:lin} 
\grad_{M} \det(D^2u) \cdot v   = {\trace}\left( 
(D^2u)_{adj}D^2v
\right ). 
\eq

\begin{example}
In two dimensions we obtain
\[ 
\grad_M \det(D^2u) v = 
u_{xx} v_{yy} + u_{yy} v_{xx} - 2 u_{xy} v_{xy},
%\DD{u}{y}\DD{(\cdot)}{x} + \DD{u}{x}\DD{(\cdot)}{y} - 2 {\frac{\partial^2 u}{\partial x \partial y}}{\frac{\partial^2 (\cdot)}{\partial x \partial y},}
\]
which is homogeneous of order one in $D^2u$.
In dimension $d \ge 2$, the linearization is homogeneous order $d-1$ in $D^2u$.
\end{example}

The linear operator
\[
L[u] \equiv \trace{ (A(x) D^2 u )}
\]
is \emph{elliptic} if the coefficient matrix $A(x)$ is positive definite.

\begin{lemma}\label{lem:linell}
Let $u \in C^2$.  The linearization of the \MA operator,~\eqref{eq:lin}, is elliptic if $D^2u$ is positive definite or, equivalently, if $u$ is (strictly) convex.
\end{lemma}

\begin{remark}
When the function $u$ fails to be strictly convex, the linearization can be degenerate elliptic, which affects the conditioning of the linear system~\eqref{eq:lin}.  When the function $u$ is non-convex, the linear system can be unstable.
\end{remark}

The definition of a nonlinear elliptic PDE operator, which follows, generalizes the definition of a linear elliptic operator.  In addition, it allows for the operator to be non-differentiable. 
\begin{definition}
Let $F(M)$ be a continuous  function defined on symmetric matrices, which we regards as a PDE operator.
Then $F(M)$ is \emph{elliptic} if it satisfies the monotonicity condition
\[ 
F(M) \leq F(N) \text{ whenever } M\leq N, 
\]
where for symmetric matrices $M\le N$ means $x^T M x \le x^T N x$ for all $x$.  
\end{definition}

\begin{example}
The operator ${\det}^+(M)$ is a non-decreasing function of the eigenvalues, so it is elliptic.
\end{example}

% !TEX root = NMATmain.tex

\section{Convergent discretization of the \MA equation}\label{sec:discretize}
In this section, we produce a discretization of the \MA equation in two and higher dimensions, and prove that solutions of the scheme converge to the viscosity solution.
This discretization is different from the one in~\cite{ObermanEigenvalues}, which was restricted to two dimensions.

We also prove convergence of Newton's method to the solution of the scheme.
To this end, we modify the original monotone  scheme from a non-differentiable scheme to a regularized (but still monotone) scheme.

\subsection{The standard finite difference discretization}
We begin by discussing the standard finite difference discretization of the \MA equation.   This is obtained by simply discretizing each of the operators using standard finite differences as in, for example,~\cite{BeFrObMA}.

Since this discretization alone does not enforce the convexity condition~\eqref{convex}, it can lead to instabilities.  
In particular, Newton's method can become unstable if this discretization is used.

There is no reason to assume that the standard discretization converges.
In fact, the two dimensional scheme has multiple solutions.  In~\cite{BeFrObMA} this discretization was used, but the solvers were designed to select the convex solution. 
In addition, the solution methods used in~\cite{BeFrObMA} do not generalize to three and higher dimensions: enforcing convexity of the solution is much more difficult when there are more that two possible eigenvalues of the Hessian.

\subsection{Eigenvalues of the Hessian in two dimensions}\label{sec:mon2D}
In two dimensions, the largest and smallest eigenvalues of a symmetric matrix can be represented by the Rayleigh-Ritz formula
\[
\lambda_1(A) = \min_{\norm{\nu} = 1} \nu^T A \nu, 
\qquad 
\lambda_2(A) = \max_{\norm{\nu} = 1} \nu^T A \nu.
\]
This formula was used in~\cite{ObermanEigenvalues} to build a monotone scheme for the~\MA operator, which is the product of the eigenvalues of the Hessian.

In higher dimensions, the formula above does not generalize naturally.  Instead, we use another characterization, which applies to positive definite matrices.

\subsection{A variational characterization of the determinant}
In this section we establish a matrix analysis result that is used to build a monotone discretization of the \MA operator.  
Our result below is closely related to \emph{Hadamard's inequality}, which states that, for a symmetric positive definite matrix $A$,
\[
\det(A) \le \prod_{i=1}^n a_{ii},
\]
with equality when $A$ is diagonal.

Consider an arbitrary symmetric positive definite matrix $A$.   We can characterize the determinant of $A$ as follows.
\begin{lemma}[Variational characterization of the determinant]\label{thm:eigs}
Let $A$ be a $d \times d$ symmetric positive definite matrix with eigenvalues $\lambda_j$ and let $V$ be the set of all  orthonormal bases of $\R^d$:
\[ V = \{(\nu_1,\ldots,\nu_d)  \mid \nu_j\in\R^d,\nu_i\perp\nu_j \text{ if }i\neq j, \|\nu_j\|_2 = 1\}. \]
Then the determinant of $A$ is equivalent to
\[ \prod\limits_{j=1}^d \lambda_j = \min\limits_{(\nu_1,\ldots,\nu_d)\in V} \prod\limits_{j=1}^d \nu_j^T A \nu_j. \]
\end{lemma}

\begin{proof}

Since $A$ is symmetric and positive definite, we can find a set of $d$ orthonormal eigenvectors $v_j$.

Any $(\nu_1,\ldots,\nu_d)\in V$, can be expressed as a linear combination of the eigenvectors:
\[ \nu_j = \sum\limits_{k=1}^d c_{jk}v_k = \sum\limits_{k=1}^d (\nu_j^Tv_k) v_k. \]
Since the $\nu_j$ and $v_j$ are both orthonormal, we can make some claims about the coefficients $c_{jk}$.
\begin{align*}
\sum\limits_{k=1}^d c_{jk}^2
  = \left(\sum\limits_{k=1}^d c_{jk}v_k^T\right)\left(\sum\limits_{l=1}^d c_{jl}v_l\right) = \nu_j^T\nu_j = 1,
\end{align*}
\[ \sum\limits_{j=1}^d c_{jk}^2
   = v_k^T\left(\sum\limits_{j=1}^d\nu_j\nu_j^T\right)v_k
   = v_k^Tv_k
   = 1.
\]
We can use these results to compute
\begin{align*}
\log \prod\limits_{j=1}^d \nu_j^TA\nu_j
  &= \sum\limits_{j=1}^d \log  \left(\nu_j^T A\nu_j \right)
  = \sum\limits_{j=1}^d \log\left(\sum\limits_{k=1}^d c_{jk}^2\lambda_k\right).
\end{align*}
Using Jensen's inequality, we conclude that
\begin{align*}
\log \prod\limits_{j=1}^d \nu_j^TA\nu_j
  &\geq \sum\limits_{j=1}^d\sum\limits_{k=1}^d c_{jk}^2 \log(\lambda_k)\\
  &= \sum\limits_{k=1}^d \left(\sum\limits_{j=1}^d c_{jk}^2\right) \log(\lambda_k)
  = \log\prod\limits_{j=1}^d \lambda_j.
\end{align*}
Since the logarithmic function is increasing, we conclude that
\[ \prod\limits_{j=1}^d \nu_j^TA\nu_j \geq  \prod\limits_{j=1}^d \lambda_j\]
with equality if the $\nu_j$ are identical to the eigenvectors $v_j$.
Thus we conclude
\[ \prod\limits_{j=1}^d \lambda_j =  \min\limits_{(\nu_1,\ldots,\nu_d)\in V} \prod\limits_{j=1}^d \nu_j^T A \nu_j.\qedhere\]
\end{proof}

We use Lemma~\ref{thm:eigs} to characterize the determinant of the Hessian of a convex $C^2$ function $\phi$.  
We can express this in terms of second directional derivatives of $\phi$ as follows:
\[ 
\det(D^2\phi) =  \min\limits_{(\nu_1,\ldots,\nu_d)\in V} \prod\limits_{j=1}^d \nu_j^T D^2\phi \, \nu_j
  = \min\limits_{(\nu_1,\ldots,\nu_d)\in V} \prod\limits_{j=1}^d \DD{\phi}{\nu_j}.
\]
The convexified \MA operator~\ref{MAplus} is represented by
\[ 
{\det}^+(D^2 \phi) =  \min\limits_{\{\nu_1\ldots\nu_d\}\in V}
\prod\limits_{j=1}^{d} 
\left (\frac{\partial^2 \phi}{\partial\nu_j^2}\right )^+
.
\]

\begin{remark}
This characterization of the \MA operator remains valid even when $\phi$ is not strictly convex, in which case the value is zero.
\end{remark}

\subsection{Wide stencil schemes}
Wide stencil schemes are needed to build consistent, monotone discretizations of degenerate second order PDEs ~\cite{Zidani, ObermanMC,ObermanIL}.
Wide stencil schemes were built for the two dimensional \MA equation in~\cite{ObermanEigenvalues}.   A wide stencil discretization of the convex envelope was given in~\cite{ObermanCENumerics}. 

To discretize the \MA operator on a finite difference grid, we approximate the second derivatives by centered finite differences; this is the \emph{spatial} discretization, with parameter $h$.
In addition, we consider a finite number of possible directions $\nu$ that lie on the grid; this is the \emph{directional} discretization, with parameter $d\theta$.
We denote this set of orthogonal vectors by $\G$.  Then we can discretize the convexified \MA operator as
\bq\label{MAmon}\tag*{$(MA)^{M}$}
\mao^{M}[u] \equiv \min\limits_{\{\nu_1\ldots\nu_d\}\in \G}
\prod\limits_{j=1}^{d} \left (\Dt_{\nu_j\nu_j}u\right)^+
\eq
where $\Dt_{\nu\nu}$ is the finite difference operator for the second directional derivative in the direction $\nu$, which lies on the finite difference grid.  
These directional derivatives are discretized by simply using finite differences on the grid:
\[ \Dt_{\nu\nu}u_i = \frac{1}{\abs{\nu}^2h^2}\left(u(x_i + \nu h) + u(x_i - \nu h) - 2u(x_i)\right). \]
Depending on the direction of the vector $\nu$, this may involve a wide stencil.  At points near the boundary of the domain, some values required by the wide stencil will not be available~(\autoref{fig:stencil}).  In these cases, we use interpolation at the boundary to construct a (lower accuracy) stencil for the second directional derivative; see~\cite{ObermanEigenvalues} for more details.

\begin{figure}
  \centering
  \subfigure[In the interior.]{
  \includegraphics[height=0.4\textwidth]
  {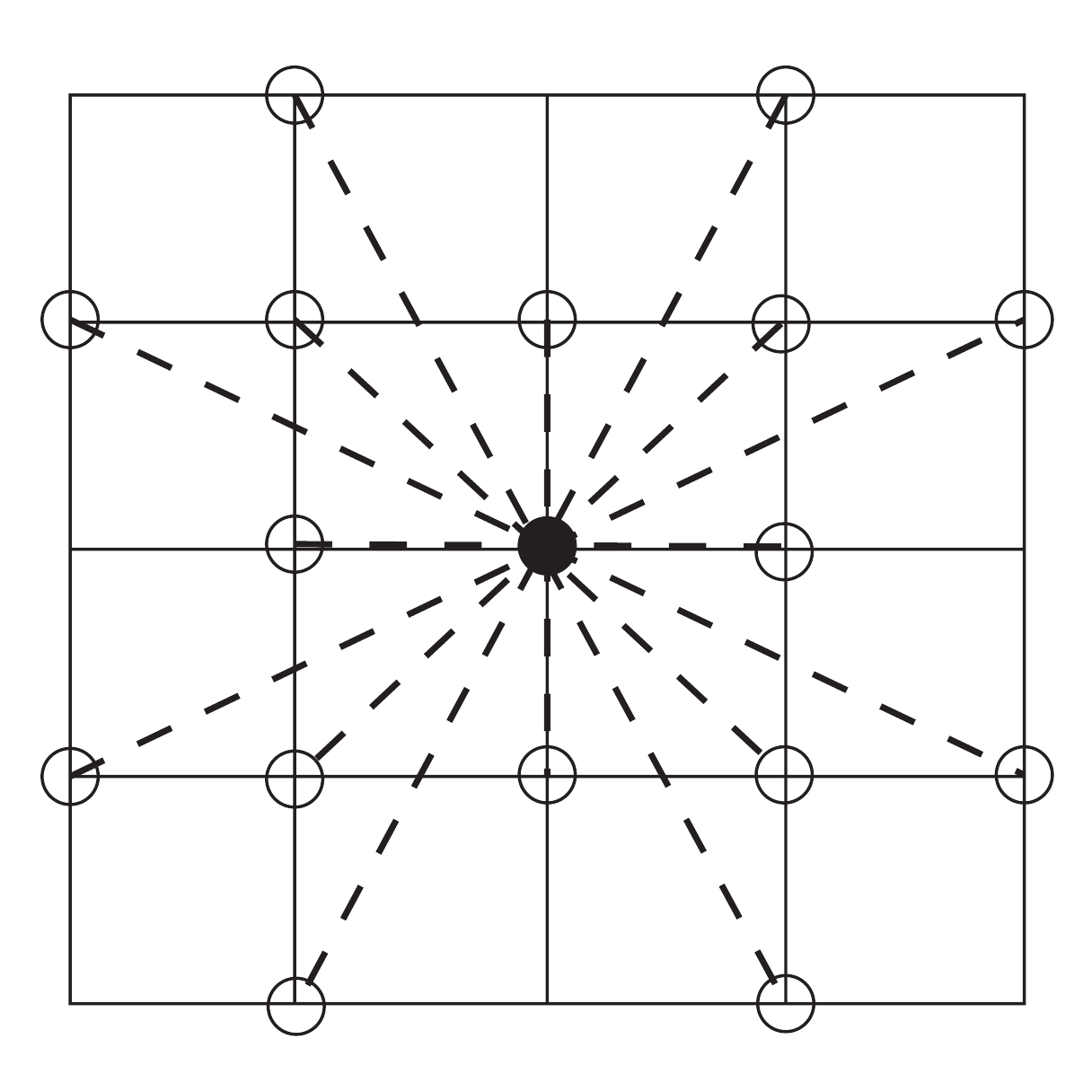}
  }                
  \subfigure[Near the boundary.]{
  \includegraphics[height=0.4\textwidth]
  {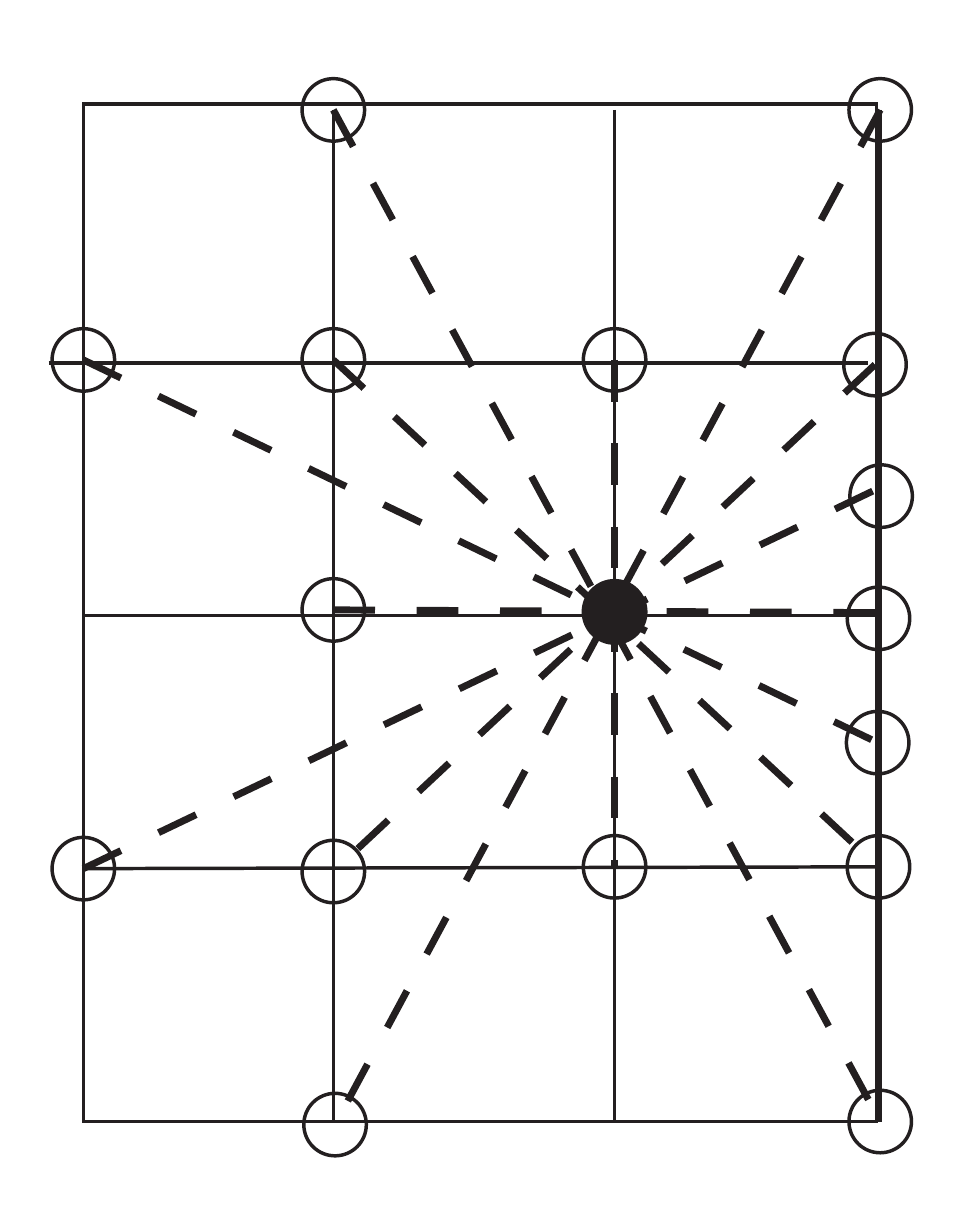}
  }
\caption{
Wide stencils on a two-dimensional grid.
}
\label{fig:stencil}
\end{figure}

Since the discretization considers only a finite number of directions $\nu$, there will be an additional term in the consistency error coming from the directional resolution $d\theta$ of the stencil.  This angular resolution will decrease and approach zero as the stencil width is increased.
In practice, we use fixed relatively narrow stencils for most computations.  However, for singular solutions the directional resolution error can dominate.

We also note that the discretization we have just described is genuinely distinct from the two-dimensional wide stencil discretization described in~\cite{ObermanEigenvalues}.  For example, we can consider the function $u(x,y) = x^2+y^2+x^2y^2$  and discretize the \MA operator using a 9-point stencil.  This allows us to choose from the set of directions
\[ \{(1,0),\,(0,1),\,(1,1),\,(1,-1)\}. \]
Using the two-dimensional characterization of the \MA equation (recalled in \autoref{sec:mon2D}), the monotone discretization produces
\[ \left(\min\limits_{\nu_1} \Dt_{\nu_1\nu_1}u\right)\left(\max\limits_{\nu_2} \Dt_{\nu_2\nu_2}u\right) = 2(2+h^2).\]
On the other hand, our new discretization has the value
\[ \min\limits_{\nu_1\perp\nu_2}\left\{\Dt_{\nu_1\nu_1}u\,\Dt_{\nu_2\nu_2}u\right\} = 4. \]

\subsection{Regularized monotone discretization}\label{sec:regularized}
The monotone discretization we have just described may not be differentiable at points where the minimum is attained along more than one direction $\nu$, or at points where the value is zero.  Since we need to differentiate the operator when we build fast solvers using Newton's method in \autoref{sec:newton}, we wish to regularize this discretization.
To ensure convergence to the viscosity solution, we need to make the regularization monotone.

One way to do this is to notice that the non-differentiability of~\ref{MAmon} 
arises only from the operations of $\max$ and $\min$.  Thus if we regularize each of these operations in a monotone way, we can reconstruct a regularized version of~\ref{MAmon} by substitution.

With that in mind, we define
\begin{align*} 
{\max}^\delta(a,b) &= \frac{1}{2}\left(a+b+\sqrt{(a-b)^2+\delta^2}\right),
\\
{\min}^\delta(a,b) &= \frac{1}{2}\left(a+b-\sqrt{(a-b)^2+\delta^2}\right). 
\end{align*}
Clearly ${\max}^\delta \to \max$ and ${\min}^\delta \to \min$ as $\delta\to 0$.
Moreover, these functions are  differentiable and  non-decreasing in each variable.

Now we can build up the regularized operator as follows.
Define
\[
\left (\Dt_{\nu\nu}u\right)^{+,\delta} = {\max}^\delta \left (\Dt_{\nu\nu}u, 0 \right ) 
\]
and replace each term in the product in~\ref{MAmon} with its regularized version as above.
Next, the minimum of the product in~\ref{MAmon} is regarded as a succession of minimums, each of which is replaced by its regularized version.

%This allows us to define the regularized second directional derivative
%\[ \Dt^\delta_{\nu_j\nu_j}u = H^\delta_+(\Dt^\delta_{\nu_j\nu_j}u,0) > 0, \]
%which is a strictly increasing function of the original directional derivative $\Dt_{\nu_j\nu_j}u$.
%
%Now we recall that the set $\G$ contains finitely many combinations of orthogonal vectors $\nu^k = \{\nu_1^k,\ldots,\nu_d^k\}$ for $k = 1,\ldots,M$.  
%We still need to regularize the minimization
%\[ \min_{\nu^k\in\G} \prod\limits_{j=1}^d \Dt^\delta_{\nu_j^k\nu_j^k}u_i. \]
%Since there are finitely many terms in the minimum, this can be accomplished one term at a time by the procedure
%\begin{align*}
% F_1^\delta[u]_i &= \prod\limits_{j=1}^d \Dt^\delta_{\nu_j^1\nu_j^1}u_i, \\
% F_k^\delta[u]_i &= H^\delta_-\left(F_{k-1}^\delta[u_i],\prod\limits_{j=1}^d \Dt^\delta_{\nu_j^k\nu_j^k}u_i\right),\,\,2\leq k \leq M. 
%\end{align*}

The resulting discretization is denoted
\bq\label{MAreg}\tag*{$(MA)^\delta$}
\mao^\delta[u]. %\equiv F^\delta_M[u]_i  
\eq
It is a smooth function of $u_i$, strictly increasing in each of the $\Dt_{\nu_j^k\nu_j^k}u_i$, and converges to the original discretization~\ref{MAmon} as $\delta \to 0$.

\subsection{Convergence of finite difference schemes}
\label{sec:parabolic}
In order to prove convergence of the solutions of our finite difference schemes to the unique viscosity solution of~\eqref{MA}, we will rely on a framework developed by Barles and Souganidis~\cite{BSnum} and extended in~\cite{ObermanSINUM}.

The framework of~\cite{BSnum} gives easily verified conditions for when approximation schemes converge to the unique viscosity solution of a PDE.  %It relies on the fact that viscosity solutions are stable under perturbations of the operator, provided the perturbed operator is also elliptic.  In this setting the consistent finite difference scheme can be regarded as a perturbed operator.  

\begin{theorem}[Convergence of Approximation Schemes]
Consider a degenerate elliptic equation, for which there exist unique viscosity solutions.  A consistent, stable approximation scheme converges uniformly on compact subsets to the viscosity solution, provided it is monotone.
\end{theorem}

While the previous theorem gives conditions for convergence, it does not provide a method for building monotone schemes or verifying when the schemes are monotone.
In~\cite{ObermanSINUM}, a framework for building and verifying the monotonicity of finite difference schemes was established. 

We recall that a finite difference equation has the form
\[ F^i[u] = F^i(u_i,u_i-u_j|_{j \neq i}). \]
This allows us to characterize a degenerate elliptic (monotone) scheme as follows~\cite{ObermanSINUM}.

\begin{definition}
The scheme $F$ is \emph{degenerate elliptic} if $F^i$ is non-decreasing in each variable.
\end{definition}

We recall Theorem~3 from~\cite{ObermanSINUM}, which provides a simple way of verifying both monotonicity and stability.

\begin{theorem}[Verifying Monotonicity and Stability]
A scheme is monotone and non-expansive in the $\ell^\infty$ norm if and only if it is degenerate elliptic.
\end{theorem}

The notion of degenerate ellipticity is also enough to guarantee the existence of a unique solution to a scheme, as proved in Theorem~8 of ~\cite{ObermanSINUM}.

\begin{theorem}[Existence and Uniqueness of Solutions for elliptic schemes]\label{thm:euschemes}
A proper, locally Lipschitz continuous degenerate elliptic scheme has a unique solution which is stable in the $\ell^\infty$ norm.
\end{theorem}

Given these general results,  the work in proving that a discretization of~\eqref{MA} converges is reduced to verifying 
two conditions: consistency and degenerate ellipticity.

\begin{remark}[Convergence rates]
While the formal accuracy of the scheme can be determined by Taylor series applied to smooth test functions, the convergence theorem guarantees only uniform convergence.  This is to be expected since, in general, viscosity solutions can be singular.  The power of the convergence result is that it applies even in the singular case.   
In general, the observed accuracy depends on both the regularity of the solution and the choice of discretization, with observed values going from $\bO(h^2)$ (for $C^4$ solutions) to $\bO(\sqrt{h})$; see~\cite{BeFrObMA}.
\end{remark}

\subsection{Proof of convergence}\label{sec:monotoneconverge}

\begin{theorem}[Convergence to Viscosity Solution]\label{thm:viscosity}
Let the PDE~\eqref{MA} have a unique viscosity solution.
Solutions of the schemes~\ref{MAmon}, \ref{MAreg} %have unique solutions which 
converge to the viscosity solution of~\eqref{MA} as $h,d\theta,\delta\to0$.
\end{theorem}
\begin{proof}
The convergence follows from verifying consistency and degenerate ellipticity as in~\cite{ObermanEigenvalues}.   This is accomplished in Lemmas~\ref{lem:deg}-\ref{lem:consistency}.
\end{proof}

Stability in $\ell^\infty$ for solutions of the schemes follows from the fact that the schemes are degenerate elliptic and the results of  Theorem~\ref{thm:euschemes}.

\begin{lemma}[Degenerate Ellipticity]\label{lem:deg}
The finite difference schemes given by~\ref{MAmon} and~\ref{MAreg} are degenerate elliptic.
\end{lemma}

\begin{proof}
From their definition, the discrete second directional derivatives $\Dt_{\nu \nu}u$ are non-decreasing functions of the differences between neighboring values and reference values, $u_j - u_i$, where $u_j$ is one of the neighboring values in the direction $\nu$.
The scheme~\ref{MAmon} is a non-decreasing combination of the operators $\min$ and $\max$ applied to the degenerate elliptic terms $\Dt_{\nu \nu}u$, so it is also degenerate elliptic.

We recall from the construction of the scheme in \autoref{sec:regularized} that the regularized scheme~\ref{MAreg} comes from replacing the operations of $\min$ and $\max$ in~\ref{MAmon} by a non-decreasing regularization of these operations.
Thus the combined scheme is also degenerate elliptic.
\end{proof}

We also require the schemes~\ref{MAreg} and~\ref{MAmon} to be consistent with the \MA equation.  We prove consistency of~\ref{MAreg} since consistency of~\ref{MAmon} is a special case.

\begin{definition}
The scheme $MA^{h,d\theta,\delta}$ is consistent with the equation~\eqref{MA} at $x_0$ if for every twice continuously differentiable function $\phi(x)$ defined in a neighborhood of $x_0$,
\[ MA^{h,d\theta,\delta}[\phi](x_0) \to MA[\phi](x_0) \text{ as } h,d\theta,\delta \to 0. \]
The global scheme defined on $\Omega$ is consistent if this limit holds uniformly for all $x \in\Omega$.
\end{definition}

\begin{lemma}\label{lem:consistency}
Let $x_0\in\Omega$ be a reference point on the grid and $\phi(x)$ be a twice continuously differentiable function that is defined and convex in a neighborhood of the grid.  Then the scheme $\mao^\delta[\phi]$ defined in~\ref{MAreg} approximates the PDE $MA[\phi]$ with accuracy
\[ \mao^\delta[\phi](x_0) = MA[\phi](x_0) + \bO(h^2 + d\theta + r(\delta)) \]
where the function $r(\delta)$ converges to zero as $\delta \to 0$.
\end{lemma}

\begin{proof}
From a simple Taylor series computation we have
\[ \Dt_{\nu\nu}\phi(x_0) = \phi_{\nu\nu}(x_0) + \bO(h^2). \]
We also recall that in \autoref{sec:regularized} we regularized the second directional derivatives to obtain
\[ \Dt^\delta_{\nu\nu}\phi(x_0) = \max\{\Dt_{\nu\nu}\phi(x_0),0\} + \bO(\delta) = \phi_{\nu\nu}(x_0) + \bO(h^2+\delta). \]
Here the maximum has no effect since we are considering convex $\phi$.

Now we recall that the \MA operator can be expressed as
\[ 
\min\limits_{\nu\in V}\prod\limits_{j=1}^d u_{\nu_j\nu_j} = \prod\limits_{j=1}^d u_{v_jv_j},
\]
where the $v_j$ are orthogonal unit vectors, which may not be in the set of grid vectors $\G$.  We can then choose a set of vectors 
\[
\frac{v + dv}{\abs{v+dv}} \in \G,
\]
so that each remainder $\abs{dv_j} = \bO(d\theta)$.

Now we consider the discretized problem
\begin{align*}
\min\limits_{\nu\in\G}\prod\limits_{j=1}^d \Dt^\delta_{\nu_j\nu_j}\phi(x_0) 
  &= \min\limits_{\nu\in\G}\prod\limits_{j=1}^d \Dt_{\nu_j\nu_j}\phi(x_0)  + \bO(\delta)\\
  &\leq \prod\limits_{j=1}^d \Dt_{(v_j+dv_j)(v_j+dv_j)}\phi(x_0)  + \bO(\delta)\\
  &=\prod\limits_{j=1}^d \frac{(v_j+dv_j)^TD^2\phi(x_0)(v_j+dv_j)}{\abs{v_j+dv_j}^2} + \bO(h^2 + \delta)\\
  &= \prod\limits_{j=1}^dv_j^TD^2\phi(x_0)v_j + \bO(h^2+d\theta+\delta)\\
  &= \min\limits_{\nu\in V}\prod\limits_{j=1}^d \phi_{\nu_j\nu_j}(x_0) + \bO(h^2+d\theta+\delta).
\end{align*}

In addition, since the set of grid vectors $\G$ is a subset of the set of all orthogonal vectors $V$, we find that
\begin{align*}
\min\limits_{\nu\in\G}\prod\limits_{j=1}^d \Dt^\delta_{\nu_j\nu_j}\phi(x_0) 
  &\geq \min\limits_{\nu\in V}\prod\limits_{j=1}^d\Dt_{\nu_j\nu_j}\phi(x_0) + \bO(\delta)\\
  &= \min\limits_{\nu \in V}\prod\limits_{j=1}^d \phi_{\nu_j\nu_j}(x_0) + \bO(h^2+\delta).
\end{align*}

We conclude that
\[ \min\limits_{\nu\in\G}\prod\limits_{j=1}^d \Dt^\delta_{\nu_j\nu_j}\phi(x_0) = \min\limits_{\nu\in V}\prod\limits_{j=1}^d \phi_{\nu_j\nu_j}(x_0) + \bO(h^2+d\theta+\delta),\]
which is precisely the characterization of the \MA operator given in~\autoref{thm:eigs}

Finally, we replace the minimum in the above scheme with a smooth function, which converges uniformly as $\delta\to0$ by construction (\autoref{sec:regdisc}).  Thus the resulting scheme will satisfy
\[ \mao^\delta[\phi](x_0) = \det(D^2 \phi(x_0)) + \bO(h^2 + d\theta + r(\delta)). \qedhere\]

\end{proof}

% !TEX root = NMATmain.tex

\section{A semi-implicit solution method}
Any  discretization of~\eqref{MA} leads to a system of nonlinear equations that must be solved in order to obtain the approximate solution.   
Newton's method requires a good initial value to converge.

In this section, we describe a semi-implicit solution method for~\eqref{MA}.
One iteration of this method will be used to compute the initial value for Newton's method.

First we describe the fully explicit method.

\subsection{Explicit solution methods for monotone schemes}
Using a monotone discretization $F[u]$ of the \MA operator,  
the fixed point iteration,
\[
u^{n+1} = u^n + dt (F[u^n]-f) 
\]
is a contraction in $\ell^\infty$, provided $dt$ is small enough.
This iteration corresponds to solving the parabolic version of the equation using a forward Euler discretization.

Explicit iterative methods have the advantage that they are simple to implement, but the number of iterations required suffers from the well known CFL condition (which applies in a nonlinear form to monotone discretizations, as explained in~\cite{ObermanSINUM}).  
This approach is slow because for stability it requires a small time step $dt$, which decreases with the spatial resolution $h$.  The time step, which can by computed explicitly, is $\bO(h^2)$. This was the approach used in~\cite{ObermanEigenvalues}.
 
\subsection{A semi-implicit solution method}\label{sec:SI}
The next method we discuss is a semi-implicit method, which involves solving a Poisson equation at each iteration.

In~\cite{BeFrObMA} we used the identity~\eqref{eq:poisson} to build a semi-implicit solution method.  We showed that the method is a contraction, but the strictness of the contraction requires strict positivity of $f$.  In practice, this meant that the iteration was fast for regular solutions, but became slower than the explicit method when $f$ was zero in large parts of the domain.

%The conditioning of the linearized equation~\eqref{eq:lin}, which affects solution time, depends on the strict convexity of the solution, see lemma~\ref{lem:linell}.  The convexity, in turn depends of strict positivity of $f$, see~\autoref{sec:regularity}.
%This explains why solution time of the semi-implicit solver depends on regularity. 

We begin with the following identity for the Laplacian in two dimensions,
\[
\abs{\Delta u} = \sqrt{(\Delta u)^2}
= \sqrt{u_{xx}^2+u_{yy}^2+2u_{xx}u_{yy}}. 
\]
If $u$ solves the \MA equation, then
\begin{align*}
\abs{\Delta u} = \sqrt{u_{xx}^2+u_{yy}^2+2u_{xy}^2+2f} = 
\sqrt{ \norm{D^2u}^2 +  2f}.
\end{align*}
This leads to a semi-implicit scheme for solving the~\MA equation, used in~\cite{BeFrObMA}. 
\bq\label{eq:poisson}
  \Delta u^{n+1} = \sqrt{2f +  \norm{D^2u^n}^2 } .
\eq
%We can use this iteration, starting with initial value, $u^{-1}$, with $D^2u^{-1} = 0$, to find $u^0$, which solves
%\[\Delta u^0 = \sqrt{2f} \]
To generalize this to  $\R^d$, we can write the Laplacian in terms of the eigenvalues of the Hessian: $\Delta u = \sum_{i=1}^d \lambda_i[D^2u]$.  
Taking the $d$-th power and expanding gives the sum of all possible products of $d$ eigenvalues
%\begin{align*}
%(\Delta u)^d
%  &= \sum\limits_{i_1,\ldots,i_d} \prod\limits_{j=1}^d \lambda_{i_j}\\
%  &= \sum\limits_{\text{distinct}\atop{ i_1,\ldots,i_d}} \prod\limits_{j=1}^d \lambda_{i_j} + remainder\\
%  &= d! \prod\limits_{i=1}^d\lambda_i + remainder\\
%  &= d! f + remainder.
%\end{align*}
\begin{align*}
(\Delta u)^d
  &= d! \prod\limits_{i=1}^d\lambda_i + P(\lambda_1, \dots, \lambda_d),
\end{align*}
where $P(\lambda)$ is a $d$-homogeneous polynomial, which we will not need explicitly.

The result is the semi-implicit scheme 
\bq\label{SImplicit}
  \Delta u^{n+1} = (d! f +   P(\lambda_1[D^2u^n], \dots, \lambda_d[D^2u^n]))^{1/d}. 
\eq
A natural initial value for the iteration is given by the solution of
\bq\label{NMinit}
\Delta u^0 = (d! f)^{1/d}.
\eq

% !TEX root = NMATmain.tex

\section{Newton's method}\label{sec:newton}
We perform the damped Newton iteration
\[ u^{n+1} = u^n  - \alpha\delu^n \]
to solve the discretized equation
\[
\mao^{M}[u] = f,
\]
where the damping parameter $\alpha$, $0<\alpha<1$, 
 is chosen  at each step to ensure that the residual $\|\mao^{M}(u^n)-f\|$ is decreasing. 

The corrector $\delu^n$ solves the linear system
\bq\label{NewtonN2}
\left ( \grad_u \mao^{M} [u^n] \right )\delu^n = \mao^{M}[u^n]-f.
\eq
Here the left hand side is our notation for the Jacobian matrix of the scheme.
The Jacobian matrix for the monotone discretization is obtained by using Danskin's Theorem~\cite{Bertsekas} and the product rule.
\[ \grad_u \mao^{M}[u] 
= \sum\limits_{j=1}^d\text{diag}\left( \prod\limits_{k \neq j}\Dt_{\nu_k^*\nu_k^*}u \right)\Dt_{\nu_j^*\nu_j^*}  \]
where the $\nu_j^*$ are the directions active in the minimum in~\ref{MAmon}.

In order for the linear equation~\eqref{NewtonN2} to be well-posed, we require
the coefficient matrix to be positive definite. 
As observed in Lemma~\ref{lem:linell}, this condition can fail if the iterate $u^n$ is not strictly convex.

\subsection{Convergence of Newton's Method for the regularized discretization}\label{sec:regdisc}
While the discretization~\ref{MAmon} is designed to be robust enough to converge to singular solutions of~\eqref{MA}, we wish to compute these solutions quickly using Newton's Method.  
Newton's Method requires that the objective function be differentiable (or at least directionally differentiable)  with an invertible Jacobian matrix.
However, the scheme~\ref{MAmon} is non-differentiable and, even at 
points of differentiability, it can have a degenerate Jacobian matrix.   The latter possibility corresponds to non-strictly convex solutions, which have degenerate Hessians. Newton's method applied to this scheme can break down for singular solutions.

The objective of this section is to show that the regularized scheme~\ref{MAreg} is designed so that even singular solutions can be computed using Newton's method.

\begin{theorem}[Newton's Method for the Discretized \MA Equation]\label{thm:newtondisc}
Suppose the PDE~\eqref{MA} has a unique viscosity solution.  Then Newton's method for the discretized system given by~\ref{MAreg} converges quadratically.
\end{theorem}

In order to prove this result, we recall a standard result on the convergence of Newton's method for a system of equations~\cite{Kelley}.  

\begin{lemma}[Newton's Method for a System of Equations]\label{thm:newton}
Consider a system of equations $F[u] = 0$ where the operator $F:\Rd\to\Rd$ and let $U \subset \Rd$ be open.  Suppose the following conditions hold:
\begin{enumerate}
\item A solution $u^* \in U$ exists.
\item $\nabla F:U\to\R^{N\times N}$ is Lipschitz continuous.
\item $\nabla F(u^*)$ is non-singular.
\end{enumerate}
Then the Newton iteration
\[ u^{n+1} = u^n - \nabla F(u^n)^{-1}F(u^n)  \]
converges quadratically to $u^*$ if $u^0\in U$ is sufficiently close to $u^*$.
\end{lemma}

\begin{remark}
In the proof below we use the fact that the discretization~\ref{MAreg} is degenerate elliptic, which leads to a positive definite and nonsingular Jacobian $\grad F$ in the Newton iteration.  
\end{remark}

\begin{proof}[Proof of Theorem~\ref{thm:newtondisc}]
For any fixed grid, the discretized system of equations has a solution, as established in Theorem~\ref{thm:viscosity}.

The scheme~\ref{MAreg} is smooth in $u$ and the Jacobian is therefore locally Lipschitz continuous.

By construction, the discrete \MA operator is strictly increasing in each of the discrete second directional derivatives (\autoref{sec:regularized}).  Thus the Jacobian will have the form
\[ \nabla_u \mao^\delta[u] = \sum\limits_{\nu^k\in\G}\sum\limits_{j=1}^{d} A_{jk}(u)\Dt_{\nu_j^k\nu_j^k} \]
where each of the $A_{jk}(u)$ is a positive definite diagonal matrix.  The Jacobian is negative definite and thus invertible.

By Theorem~\ref{thm:newton}, Newton's method converges for the discretized system~\ref{MAreg}.  \qedhere \end{proof}

\subsection{Initialization of Newton's method}\label{sec:init}
Newton's method requires a good initialization for the iteration.  Since we need the resulting linear system to be well-posed, it is essential that the initial iterate: (i) be convex, (ii) respect the boundary conditions, (iii) be close to the solution.   In addition, the computational cost of initialization should be low.

In order to achieve these conditions we use one step of the semi-implicit scheme~\eqref{SImplicit} to obtain a close initial value.  This amounts to solving~\eqref{NMinit} along with consistent Dirichlet boundary conditions~\eqref{Dirichlet}.
As necessary, we then convexify the result using the method of~\cite{ObermanCENumerics} (this last step was only needed for the most singular solution).
Since both these steps can be performed on a very coarse grid and interpolated onto the finer grid, the cost of  initialization is low.

\subsection{Preconditioning}\label{sec:reg}
In degenerate examples, the PDE for $\delu^n$~\eqref{NewtonN2} may be degenerate, which can lead to an ill-conditioned or singular Jacobian.  To get around this problem, we regularize the Jacobian to make sure the linear operator is strictly negative definite; this will not change the fixed points of Newton's method.  We accomplish this by replacing the second directional derivatives $u_{\nu\nu}$ with
\[ \tilde{u}_{\nu\nu} = \max\{u_{\nu\nu},\e\}. \]
%and the cross term $u_{xy}$ with 
%\[ \tilde{u}_{xy} = \text{sgn}(u_{xy})\sqrt{\tilde{u}_{xx}\tilde{u}_{yy}-\max\{\tilde{u}_{xx}\tilde{u}_{yy}-u_{xy}^2,\e^2\}}. \]
Here $\e$ is a small parameter.  In the computations of \autoref{sec:2d}, we take $\e = \frac{1}{2dx^2}\ex{8}$.

% !TEX root = NMATmain.tex

\section{Computational results in two dimensions.}\label{sec:2d} 
In this section, we summarize the results of a number of two-dimensional examples computed using the methods described in this paper.  In particular, we are interested in comparing the computation time for Newton's method with the time required by the methods proposed in~\cite{BeFrObMA}. 

We perform the computations using a 17 point stencil on an $N \times N$ grid on a square.  

We solved all linear systems using the MATLAB backslash operator, which performs an LU decomposition of the sparse systems arising in Newton's method.  Naturally, computations could be made faster by using a compiled programming language or a more efficient linear solver.  However, this implementation is sufficient to demonstrate the efficiency of our method.

\subsection{Four representative examples}
We have tested the monotone scheme on a number of examples of varying regularity.  To illustrate these, we present detailed results for four representative examples.

To define these exact solutions, we first write $\xv = (x,y)$ and $\xv_0 = (.5, .5)$ for the center of the domain.

The first example, which is smooth and radial, is given by
\bq\label{eq:c2} 
%u(x,y) = \exp\left(\frac{x^2+y^2}{2}\right),
u(\xv) = \exp \left( \frac{ \norm{\xv}^2}{2} \right),
\qquad 
%f(x,y) = (1+x^2+y^2)\exp(x^2+y^2).
f(\xv) = \left(1+ \norm{\xv}^2\right)\exp\left( \norm{\xv}^2 \right).
\eq
The second example, which is $C^1$, is given by
%\bq\label{eq:c1} 
%u(x,y) = \max\left\{\frac{(\sqrt{(x-1/2)^2+(y-1/2)^2}-0.2)^2}{2},0\right\}, 
%\eq
\bq\label{eq:c1} 
u(\xv) = \frac{1}{2}\left( (\norm{\xv-\xv_0} -0.2)^+\right )^2, 
\quad
f(\xv) = 
\left( 
1 - \frac{0.2}{\norm{\xv-\xv_0}}
\right)^+.
%\quad \xv_0 = (.5,.5)
\eq
%\[ 
%f(x,y) = \begin{cases}1 - \frac{0.2}{\sqrt{(x-1/2)^2+(y-1/2)^2}} & \sqrt{(x-1/2)^2+(y-1/2)^2}>0.2\\0 & \text{otherwise.}
%\end{cases} 
%\]

The third example is twice differentiable in the interior of the domain, but has an unbounded gradient near the boundary point $(1,1)$.  The solution is given by
\bq\label{eq:blowup} 
u(\xv) = -\sqrt{2-\norm{\xv}^2},
\qquad 
%f(\xv) = \frac{2}{\left(2-\norm{\xv}^2\right )^2}. 
f(\xv) = 2 {\left(2-\norm{\xv}^2\right)^{-2} }.
\eq

The final example is the cone, which was discussed in~\autoref{sec:alex}.  It is Lipschitz continuous.
\bq\label{eq:cone} 
u(\xv) = \sqrt{\norm{\xv-\xv_0}},
\qquad
f = \mu = \pi \,\delta_{\xv_0}.
%\quad  \xv_0 = (.5,.5).
\eq
In order to approximate the solution on a grid with spatial resolution $h$ using viscosity solutions, we approximate the measure $\mu$ by its average over the ball of radius $h/2$, which gives
\[ 
f^h  = 
\begin{cases}
4/h^2 &  \text{ for } \norm{\xv - \xv_0} \leq h/2,\\
0 & \text{ otherwise.}
\end{cases}
\]

\subsection{Visualization of solutions and gradient maps}\label{sec:maps}
In~\autoref{fig:solutions} the solutions and the gradient maps for the first three representative examples are presented. For example~\eqref{eq:cone}, the gradient map is too singular to illustrate.
To visualize the maps, the image of a Cartesian mesh under the mapping 
\[ 
\left(\begin{array}{c}x\\y\end{array}\right) \to \left(\begin{array}{c}\Dt_{x}{u}\\ \Dt_{y}{u}\end{array}\right) \] 
is shown, where $(\Dt_x u,\Dt_y u)$ is the numerical gradient of the solution of the \MA equation.  
The image of a circle is plotted for visualization purposes; 
the equation is actually solved on a square. For reference, the identity mapping is also displayed.  

In each case, the maps agree with the maps obtained using the gradient of the exact solution.

\begin{figure}[htdp]
	\centering
	\subfigure[]{\includegraphics[width=.4\textwidth]{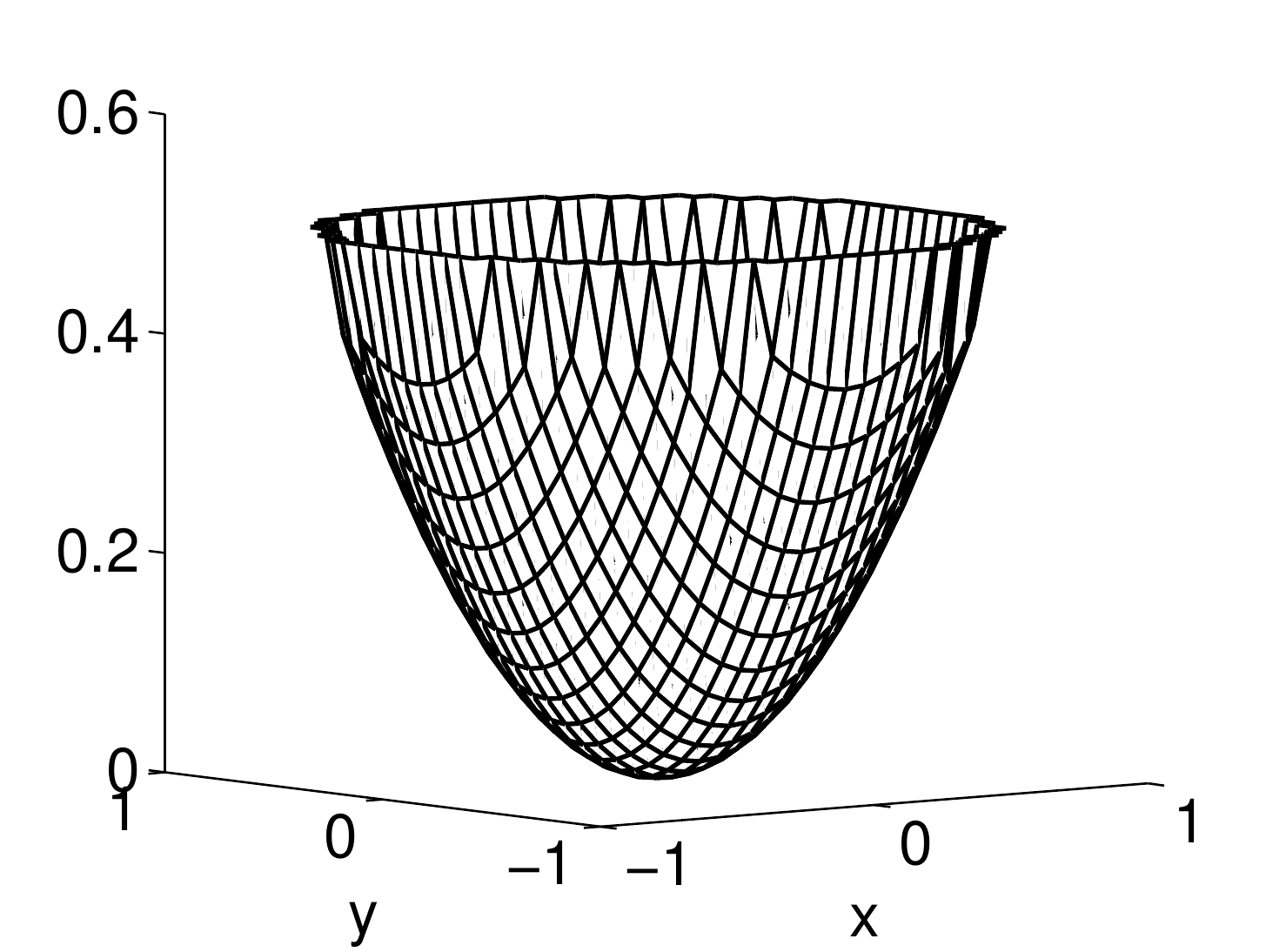}\label{fig:init_surf}}
        \subfigure[]{\includegraphics[width=.4\textwidth]{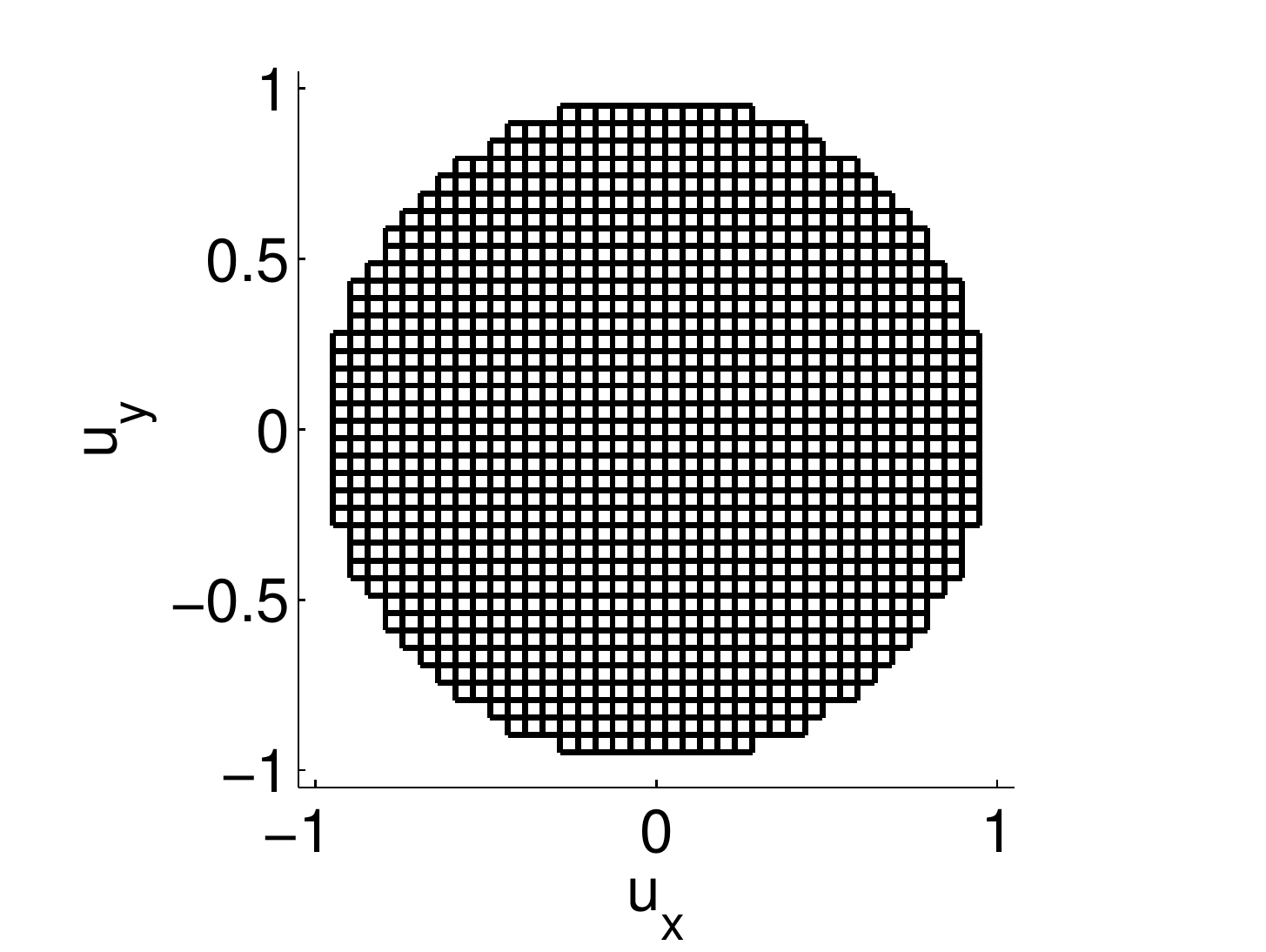}\label{fig:init_mesh}}
        \subfigure[]{\includegraphics[width=.4\textwidth]{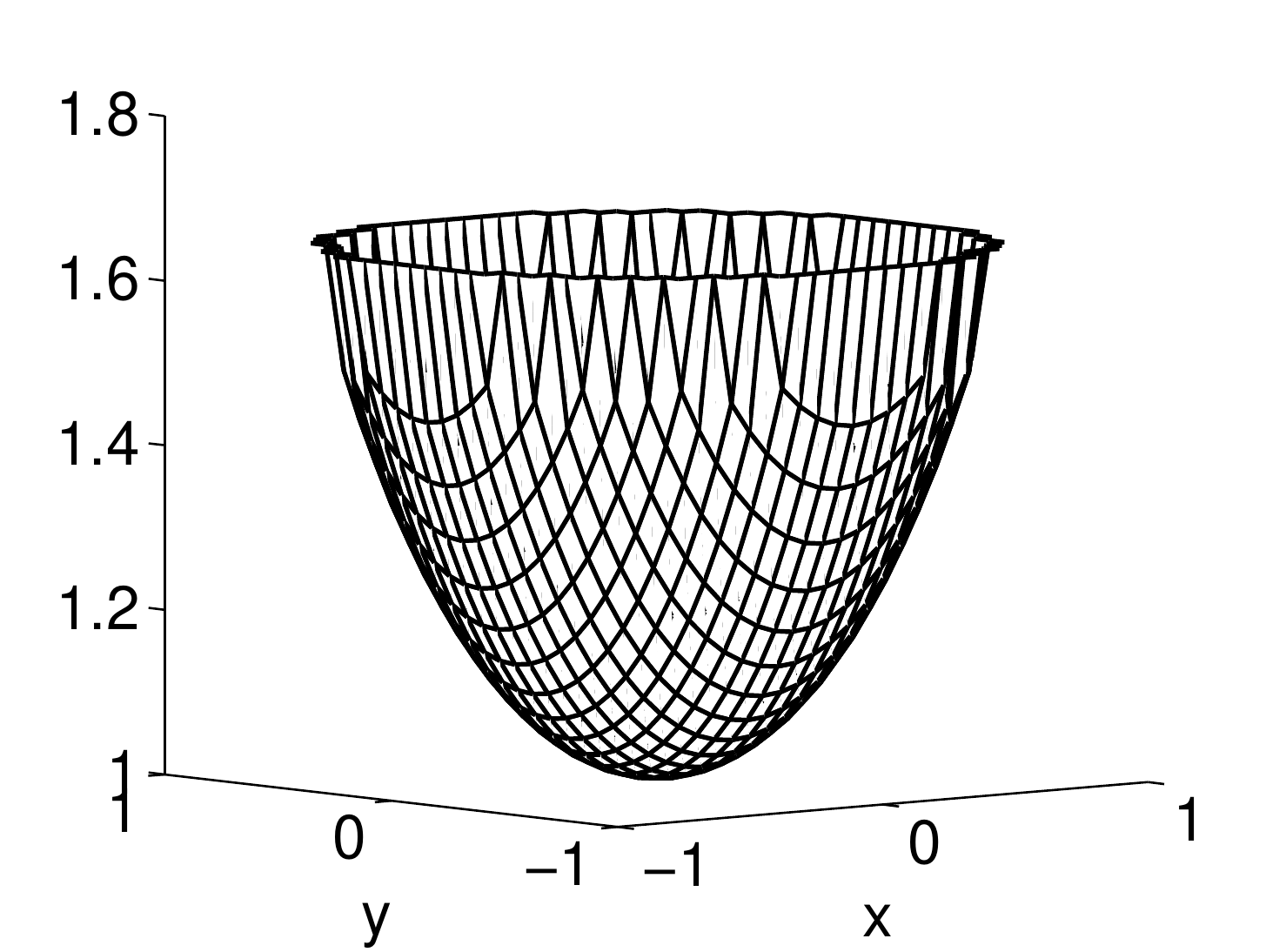}\label{fig:radial_surf}}
        \subfigure[]{\includegraphics[width=.4\textwidth]{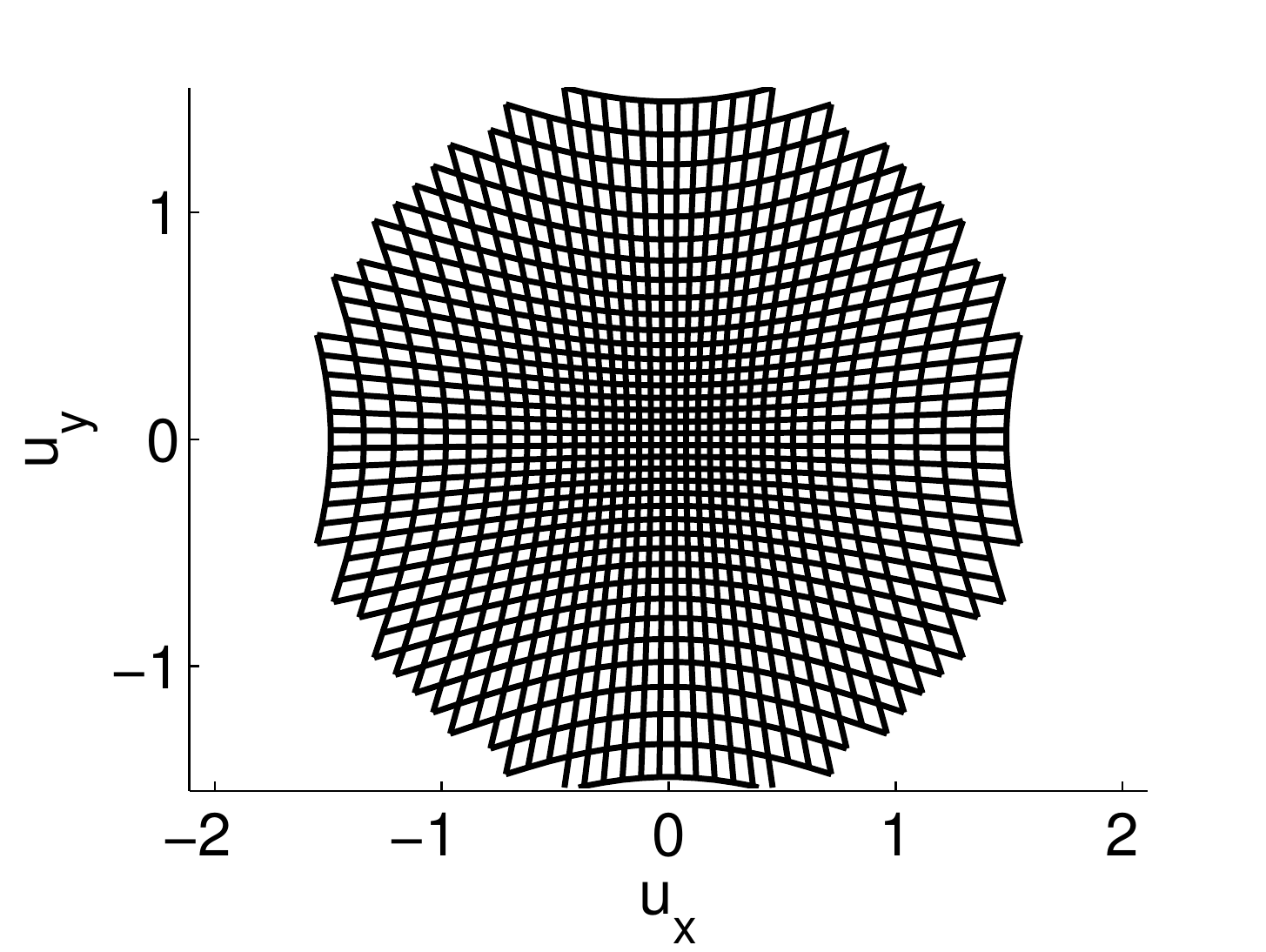}\label{fig:radial_mesh}}	
        \subfigure[]{\includegraphics[width=.4\textwidth]{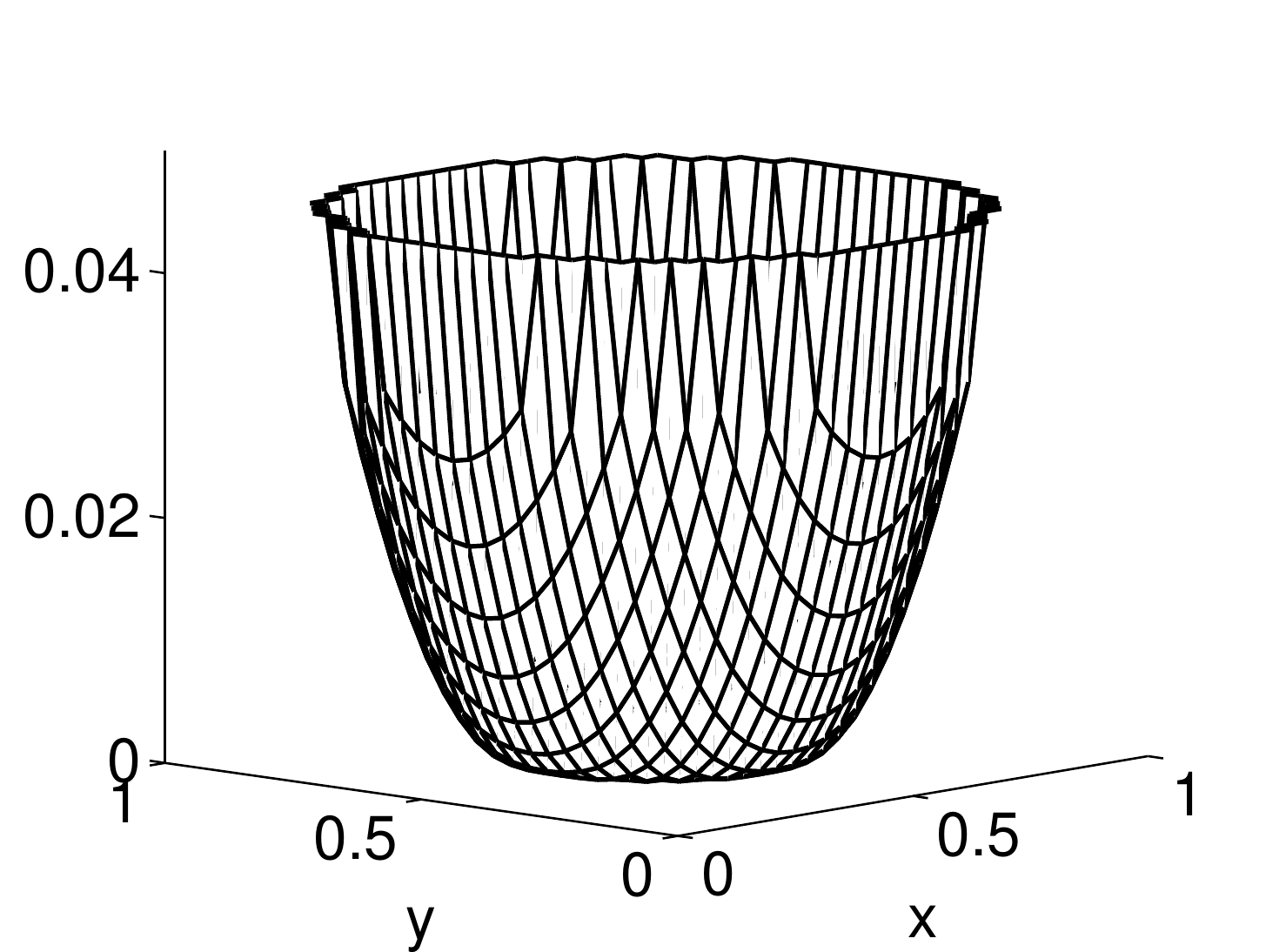}\label{fig:bowl2_surf}}
        \subfigure[]{\includegraphics[width=.4\textwidth]{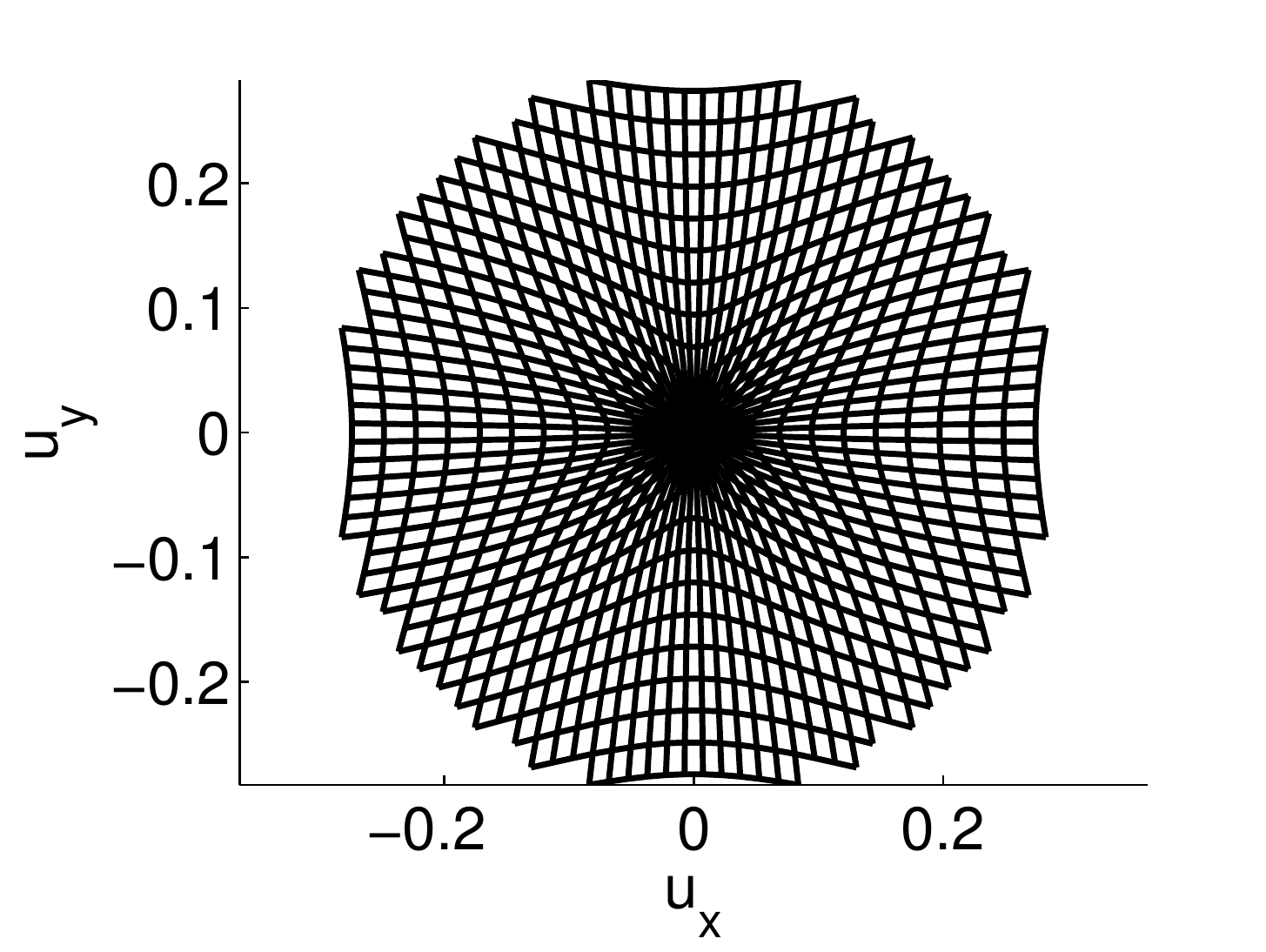}\label{fig:bowl2_mesh}}	
        \subfigure[]{\includegraphics[width=.4\textwidth]{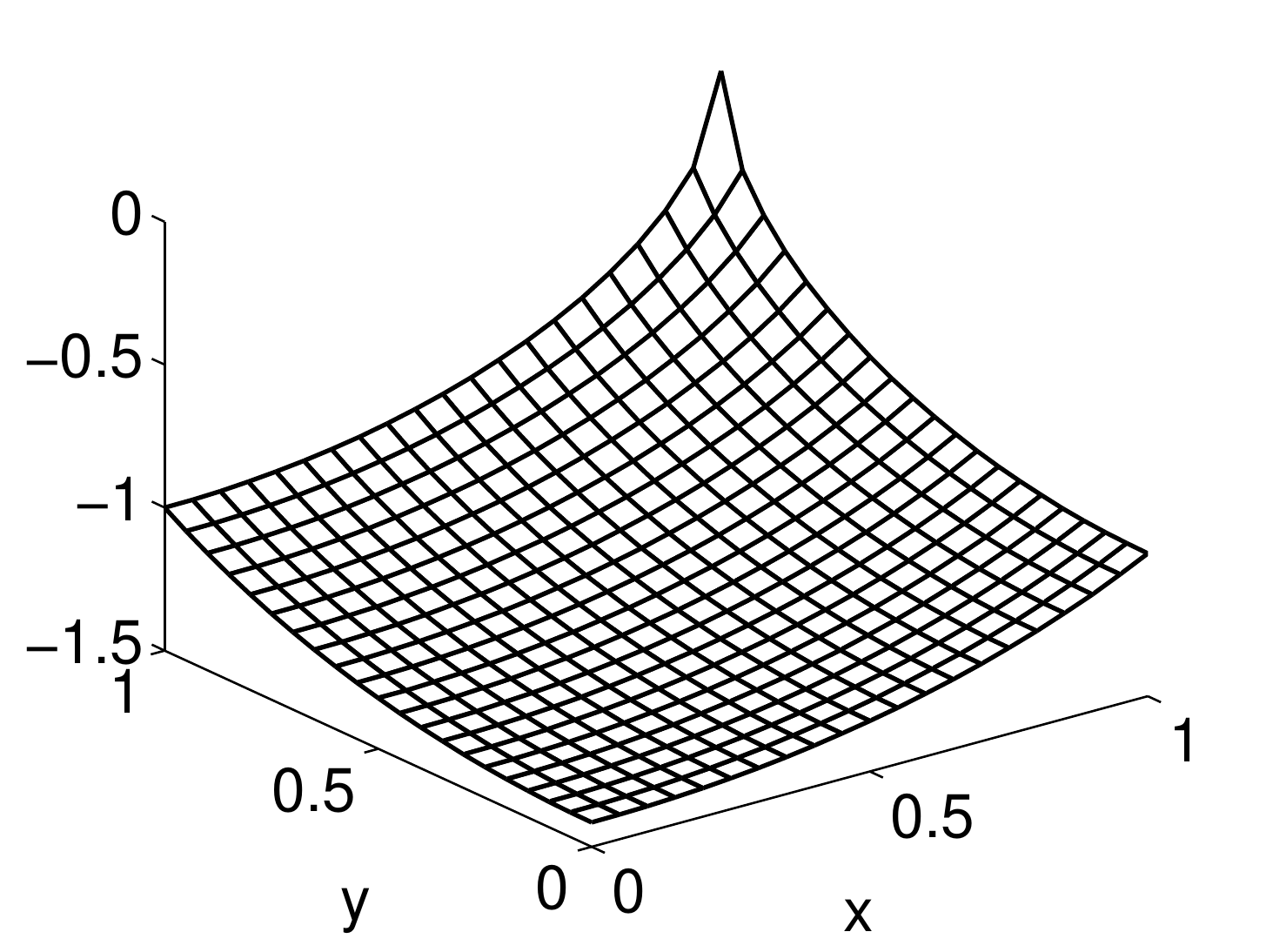}\label{fig:ball_surf}}
        \subfigure[]{\includegraphics[width=.4\textwidth]{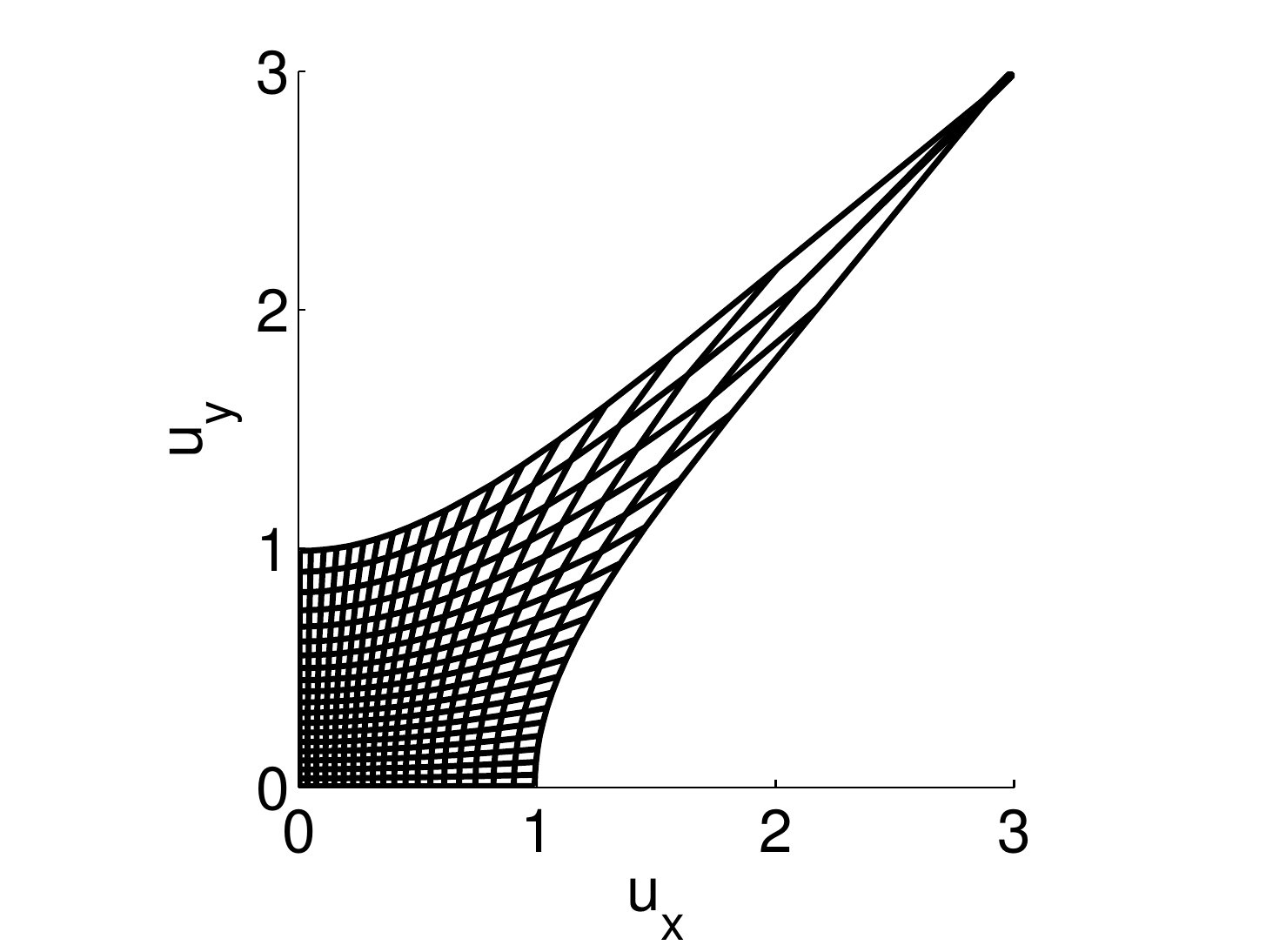}\label{fig:ball_mesh}}
  	\caption{Solutions and mappings for the \subref{fig:init_surf},\subref{fig:init_mesh} identity map, \subref{fig:radial_surf},\subref{fig:radial_mesh} $C^2$ example, \subref{fig:bowl2_surf},\subref{fig:bowl2_mesh} $C^1$ example, and \subref{fig:ball_surf},\subref{fig:ball_mesh} example with blow-up.}
  	\label{fig:solutions}
\end{figure}

\begin{table}[htdp]
\begin{tabular}{c|ccc}
 & \multicolumn{3}{c}{Regularity of Solution} \\
Method  & $C^{2,\alpha}$   &   $C^{1,\alpha}$  & $C^{0,1}$ (Lipschitz)
\vspace{.05cm}\\
\hline
Gauss-Seidel & Moderate  & Moderate &  Moderate \\
             &($\sim\bO(M^{1.8})$) &($\sim\bO(M^{1.9})$) & ($\sim\bO(M^2)$)\\
Poisson      & Fast    & Slow--Fast  & Slow \\
             & ($\sim\bO(M^{1.4}$)  & ($\sim\bO(M^{1.4})$--blow-up) & ($\sim\bO(M^2)$--blow-up)\\
Newton       & Fast     & Fast     & Fast  \\
             & ($\sim\bO(M^{1.3})$) & ($\sim\bO(M^{1.3})$) & ($\sim\bO(M^{1.3})$)\
\vspace{0.2cm}\\
\end{tabular}
\caption{Approximate computation time required by the Gauss-Seidel, Poisson, and Newton's methods for two-dimensional problems of varying regularity.  Here $M = N^2$ is the total number of grid points.}
\label{table:time}
\end{table}

\subsection{Computation time}\label{sec:time2d}
The computation times for the four representative examples are presented in \autoref{table:time2d}.
The computations time are compared to those for the Gauss-Seidel and Poisson iterations described in~\cite{BeFrObMA}.  
In each case, the Newton solver is faster in terms of absolute solution time. 

\autoref{table:time} presents order of magnitude solution times.
The order of magnitude solution time for Newton's method 
is independent of the regularity of the solution and faster than both of the other methods.
In particular, the computation time required by our Newton solver is roughly $\bO(M^{1.3})$.  This is really dependent on two different things: the number of Newton iterations required, which can increase as the grid is refined, and the speed of the linear solver.  The use of a more efficient linear solver might further reduce the order of magnitude solution time for this method.

\begin{table}[htdp]
\begin{tabular}{rrrrr}
\multicolumn{5}{c}{$C^2$ Example \eqref{eq:c2}}\\
$N$  & Newton      & \multicolumn{3}{c}{CPU Time (seconds)} \\
   & Iterations  & Newton   &   Poisson  & Gauss-Seidel
\vspace{.05cm}\\
\hline
31  &  3 & 0.2 & 0.7 & 2.2\\
%45  &  4 & 0.3 & 1.1 & 4.1\\
63  &  6 & 0.9 & 1.9 & 15.0\\
%89  &  5 & 1.4 & 4.8 & 57.6\\
127 &  7 & 4.9 & 9.6 & 236.7\\
%181 &  7 & 14.8 & 23.2 & 1004.0\\
255 &  7 & 43.5 & 52.6 & ---\\
361 &  7 & 172.8 & 162.6 & ---\\
\hline\hline\\
\multicolumn{5}{c}{$C^1$ Example \eqref{eq:c1}}\\
$N$   & Newton      & \multicolumn{3}{c}{CPU Time (seconds)} \\
   & Iterations  & Newton   &   Poisson  & Gauss-Seidel
\vspace{.05cm}\\
\hline
31  &  4 & 0.3 & 1.1 & 0.8\\
%45  &  6 & 0.4 & 6.1 & 2.8\\
63  &  7 & 0.8 & 20.5 & 9.5\\
%89  &  9 & 2.3 & 80.0 & 35.9\\
127 &  11 & 6.6 & 256.8 & 145.5\\
%181 &  13 & 20.2 & --- & 558.0\\
255 &  16 & 61.5 & --- & ---\\
361 &  20 & 350.4 & --- & ---\\
\hline\hline\\
\multicolumn{5}{c}{Example with blow-up \eqref{eq:blowup}}\\
$N$   & Newton      & \multicolumn{3}{c}{CPU Time (seconds)} \\
   & Iterations  & Newton   &   Poisson  & Gauss-Seidel
\vspace{.05cm}\\
\hline
31  &  4  & 0.3 & 0.5 & 0.8\\
%45  &  5  & 0.3 & 1.4 & 5.3\\
63  &  4 & 0.5 & 2.9 & 19.4\\
%89  &  6  & 1.6 & 8.1 & 74.1\\
127 &  5 & 3.7 & 17.7 & 293.3\\
%181 &  6 & 12.3 & 51.4 & 1637.1\\
255 &  7 & 41.4 & 128.2 & ---\\
361 &  9 & 184.1 & 374.5 & ---\\
%\vspace{0.02cm}\\
\hline\hline\\
\multicolumn{5}{c}{$C^{0,1}$ (Lipschitz) Example \eqref{eq:cone}}\\
$N$   & Newton      & \multicolumn{3}{c}{CPU Time (seconds)} \\
   & Iterations  & Newton   &   Poisson  & Gauss-Seidel \\
\vspace{.05cm}\\
\hline
31  &  9 & 0.5 & 5.3 & 0.8\\
%45  &  11 & 0.6 & 27.8 & 5.9\\
63  &  15 & 1.4 & 91.9 & 21.5\\
%89  &  22 & 4.3 & 451.0 & 90.5\\
127 &  32 & 14.1 & 1758.2 & 373.9\\
%181 &  30 & 34.6 & --- & 1576.1\\
255 &  34 & 101.7 & --- & ---\\
361 &  29 & 280.2 & --- & ---
\end{tabular}
\caption{Computation times for the Newton, Poisson, and Gauss-Seidel methods for four representative examples.}
\label{table:time2d}
\end{table}

\subsection{Accuracy}\label{sec:accuracy}
Finally, we present accuracy results for the four representative examples; see~\autoref{table:err2d}.  We perform the computations using the monotone scheme on 9, 17, and 33 point stencils.  
The accuracy of the scheme is determined by a combination of the directional resolution ($d\theta$) error and the spatial discretization error.
Widening the stencil, which has the effect of decreasing $d\theta$, improves the accuracy, as does increasing the number of grid points.
We also compared the accuracy to standard finite differences using the results of~\cite{BeFrObMA}.   (Solution times were longer using these methods).
Standard finite differences are formally more accurate since there is no $d\theta$ error.   The numerical results show that for the first two, more regular examples, the standard finite differences are more accurate.  However, for the two more singular examples, the monotone finite differences are slightly more accurate.

\begin{table}[htdp]
\begin{tabular}{rrrrr}
\multicolumn{5}{c}{{Max Error,} $C^2$ Example \eqref{eq:c2}}\\
$N$  & 9 Point   &   17 Point  & 33 Point & F.D.
\vspace{.05cm}\\
\hline
31 & $17.9\ex{4}$ & $8.9\ex{4}$& $7.0\ex{4}$ & $7.14\ex{5}$\\
%45 & $16.6\ex{4}$ & $6.0\ex{4}$& $4.3\ex{4}$ \\
63 & $16.2\ex{4}$ & $5.1\ex{4}$& $3.1\ex{4}$ & $1.73\ex{5}$ \\
%89 & $16.0\ex{4}$ & $4.8\ex{4}$& $2.2\ex{4}$ \\
127 & $15.9\ex{4}$& $4.6\ex{4}$& $1.8\ex{4}$ & $4.3\ex{6}$\\
%181 & $15.9\ex{4}$& $4.5\ex{4}$& $1.6\ex{4}$ \\
255 & $15.9\ex{4}$& $4.4\ex{4}$& $1.5\ex{4}$  & $1.1\ex{6}$\\
361 & $15.9\ex{4}$& $4.4\ex{4}$& $1.5\ex{4}$  & $0.5\ex{6}$\\
\hline\hline\\
\multicolumn{5}{c}{{Max Error,}  $C^1$ Example \eqref{eq:c1}}\\
$N$   & 9 Point   &   17 Point  & 33 Point & F.D.
\vspace{.05cm}\\
\hline
31  & $3.0\ex{3}$& $1.7\ex{3}$& $1.5\ex{3}$ &   $2.6\ex{4}$\\
%45  & $2.6\ex{3}$& $1.2\ex{3}$& $0.9\ex{3}$ \\
63  & $2.5\ex{3}$& $1.0\ex{3}$& $0.6\ex{3}$ & $1.5\ex{4}$\\
%89  & $2.3\ex{3}$& $0.8\ex{3}$& $0.5\ex{3}$ \\
127 & $2.3\ex{3}$& $0.8\ex{3}$& $0.3\ex{3}$ &  $0.6\ex{4}$ \\
%181 & $2.3\ex{3}$& $0.7\ex{3}$& $0.3\ex{3}$ \\
255 & $2.2\ex{3}$& $0.7\ex{3}$& $0.3\ex{3}$ & --- \\
361 & $2.2\ex{3}$& $0.7\ex{3}$& $0.3\ex{3}$ & ---\\
\hline\hline\\
\multicolumn{5}{c}{{Max Error,} Example with blow-up \eqref{eq:blowup}}\\
$N$   & 9 Point   &   17 Point  & 33 Point & Standard F.D.
\vspace{.05cm}\\
\hline
31  & $1.7\ex{3}$& $1.7\ex{3}$& $1.7\ex{3}$ & $17.15\ex{3}$\\
%45  & $1.0\ex{3}$& $1.0\ex{3}$& $1.0\ex{3}$ \\
63  & $0.9\ex{3}$& $0.6\ex{3}$& $0.6\ex{3}$ & $12.53\ex{3}$\\
%89  & $0.9\ex{3}$& $0.4\ex{3}$& $0.3\ex{3}$  \\
127 & $0.8\ex{3}$& $0.3\ex{3}$& $0.2\ex{3}$  & $9.00\ex{3}$\\
%181 & $0.8\ex{3}$& $0.3\ex{3}$& $0.2\ex{3}$ \\
255 & $0.8\ex{3}$& $0.3\ex{3}$& $0.2\ex{3}$ & $6.42\ex{3}$\\
361 & $0.8\ex{3}$& $0.3\ex{3}$& $0.2\ex{3}$ & $5.41\ex{3}$ \\
\hline\hline\\
\multicolumn{4}{c}{{Max Error,} $C^{0,1}$ (Lipschitz) Example \eqref{eq:cone}}\\
$N$  & 9 Point   &   17 Point  & 33 Point & Standard F.D.
\vspace{.05cm}\\
\hline
31  & $12\ex{3}$& $3\ex{3}$& $3\ex{3}$  & $10\ex{3}$ \\
%45  & $12\ex{3}$& $3\ex{3}$& $2\ex{3}$ \\
63  & $11\ex{3}$& $3\ex{3}$& $2\ex{3}$  & $6\ex{3}$ \\
%89  & $11\ex{3}$& $4\ex{3}$& $2\ex{3}$ \\
127 & $11\ex{3}$& $4\ex{3}$& $2\ex{3}$& $3\ex{3}$ \\
%181 & $11\ex{3}$& $4\ex{3}$& $2\ex{3}$ \\
255 & $11\ex{3}$& $4\ex{3}$& $1\ex{3}$ & --- \\
361 & $11\ex{3}$& $4\ex{3}$& $1\ex{3}$ & --- 
\end{tabular}
\caption{Accuracy of the monotone scheme for different stencil widths and for standard finite differences on four representative examples.}
\label{table:err2d}
\end{table}

\section{Computational results in three dimensions}\label{sec:3d}
Next, we perform computations to test the speed and accuracy of Newton's method for three-dimensional problems.
These computations are done using a 19 point stencil on an $N^3$ grid on the square $[0,1]^3$.

The methods of~\cite{BeFrObMA} were restricted to the two-dimensional \MA equation, so no computations were available for comparison in three dimensions.

As before, we  performed computations on three representative exact solutions of varying regularity.  To simplify the following expressions, we denote a vector in $\R^3$ by
\[ 
\xv = (x,y,z), 
\]
and let
\[ 
\xo = (.5, .5, .5) 
\]
be the center of the domain.
The first example solution is the $C^2$ function
\bq\label{eq:c23d} 
u(\xv) = \exp{\left( \frac{\abs{\xv}^2}{2}\right)}, 
\qquad 
f(\xv) = \left(1+\abs{\xv}^2\right)\exp{ \left(  \frac{ 3\abs{\xv}^2}2 \right)}. 
\eq
The second example solution is the $C^1$ function
\bq\label{eq:c13d} 
u(\xv) =  
 \left(   (\abs{\xv-\xo}-0.2)^{+}  \right )^2  /{2},
\eq
\[ 
f(\xv) = \begin{cases}
1 - \frac{0.4}{\abs{\xv-\xo}}+\frac{0.04}{\abs{\xv-\xo}^2}, 
 &\abs{\xv-\xo}>0.2,
\\
0 & \text{otherwise.}
\end{cases}\]
The third example is the surface of a ball, which is differentiable in the interior of the domain, but has an unbounded gradient at the boundary:
\bq\label{eq:blowup3d}
u(\xv) = -\sqrt{3-\abs{\xv}^2},
\qquad 
f(\xv) = 3\left (3-\abs{\xv}^2 \right)^{-5/2}.
\eq

Computation times and accuracy results for the three-dimensional examples are presented in~\autoref{table:3d}.

\begin{table}[htdp]
\begin{tabular}{cccc}
\multicolumn{4}{c}{$C^2$ Example \eqref{eq:c23d}}\\
$N$ & Max Error & Iterations  & CPU Time (s) 
\vspace{.05cm}\\
\hline
7  & 0.0151 & 2 & 0.1 \\
11 & 0.0140 & 3 & 0.1 \\
15 & 0.0132 & 5 & 0.5 \\
21 & 0.0127 & 6 & 3.6 \\
31 & 0.0125 & 5 & 34.7\\
\vspace{0.02cm}\\
\hline\hline\\
\multicolumn{4}{c}{$C^1$ Example \eqref{eq:c13d}}\\
$N$ & Max Error & Iterations  & CPU Time (s) 
\vspace{.05cm}\\
\hline
7 & 0.0034 & 1 & 0.02\\
11 & 0.0022 & 1 & 0.06\\
15 & 0.0019 & 1 & 0.17\\
21 & 0.0020 & 2 & 1.42\\
31 & 0.0019 & 2 & 16.70\\
\vspace{0.02cm}\\
\hline\hline\\
\multicolumn{4}{c}{Example with Blow-up \eqref{eq:blowup3d}}\\
$N$ & Max Error & Iterations  & CPU Time (s) 
\vspace{.05cm}\\
\hline
7  &  $9.6\ex{3}$ & 1 & 0.02\\
11 & $5.3\ex{3}$ & 3 & 0.12\\
15 & $4.7\ex{3}$ & 3 & 0.34\\
21 & $4.3\ex{3}$ & 6 & 3.65\\
31 & $3.9\ex{3}$ & 8 & 54.70
\end{tabular}
\caption{Maximum error and computation time on three representative three-dimensional examples.}
\label{table:3d}
\end{table}
% !TEX root = NMATmain.tex

\section{Conclusions}
A fast, convergent finite difference solver for the elliptic \MA equation was built, analyzed, and implemented.
Computational results were presented using two- and three-dimensional exact solutions of varying regularity, from smooth to non-differentiable.

A monotone discretization was built using a method which applied arbitrary dimensions. % (It generalizes an existing monotone discretization which applied only in two dimensions).
 A proof of convergence of the finite difference approximation to the unique viscosity solution of the equation was given.  
  
The discretized equations were solved using Newton's method, which is fast, experimentally  $\bO(M^{1.3})$ where $M$ is the number of data points, independent of the regularity of the solution.  The implementation of Newton's method was significantly (orders of magnitude) faster than the two other methods used for comparison.  
A close initial approximate solution was provided using a low cost method based on performing one step of a previously established solution method, then taking the convex envelope as needed.
A proof of convergence of Newton's method was also provided.

The solver presented here used a novel discretization in general dimensions, accompanied by a fast solution method.  
In terms of solution time and stability, the resulting solver is a significant improvement over existing methods for the solution of the elliptic \MA equation.  This conclusion is valid for both regular and singular solutions.   On smooth solutions the accuracy was lower than when using standard finite differences, but on singular solutions the accuracy was slightly better.   

After the present paper was submitted, we succeeded in
improving the accuracy of the discretization by applying
the monotone method in instances when stability considerations
were crucial, and using a more accurate solver otherwise
~\cite{FroeseObermanFastMA}. Similarly, we have since
extended these methods to more general \MA type equations and
have investigated the use of the various boundary conditions
that arise in the mapping problem ~\cite{FroeseMATransport}.

\bibliographystyle{plain}
\bibliography{MongeAmpere}

\end{document}